\documentclass[12pt,reqno]{amsart}
\usepackage{amsfonts,amsmath,amssymb,epsf,color,array,colortbl,amscd,bbm,tensor,tikz,comment}
\usepackage{tikz}
\usetikzlibrary{shadings}
\usetikzlibrary{matrix,arrows}
\usepackage{fullpage}
\theoremstyle{plain}
\newtheorem{thm}{Theorem}
\newtheorem{prop}[thm]{Proposition}
\newtheorem{conj}[thm]{Conjecture}
\newtheorem{cor}[thm]{Corollary}
\newtheorem{lemma}[thm]{Lemma}
\theoremstyle{definition}
\newtheorem{definition}[thm]{Definition}
\newtheorem{remark}[thm]{Remark}

\allowdisplaybreaks

\newcommand{\field}[1]{\mathbb{#1}}

\newcommand{\Z}{\field{Z}}

\newcommand{\C}{\field{C}}

\newcommand{\sgn}{\operatorname{sgn}}

\newcommand{\im}{\operatorname{Im}} 
\newcommand{\re}{\operatorname{Re}} 
\newcommand{\img}{\operatorname{im}} 
\newcommand{\tr}{\operatorname{tr}} 
\newcommand{\qdim}[1]{\operatorname{qdim}\left[#1\right]}
\renewcommand{\hom}{\operatorname{Hom}}
\newcommand{\dd}{{d}} 
\newcommand{\SmatMM}[2]{S^{\operatorname{Vir}}_{(#1),(#2)}} 
\newcommand{\SmatFT}[3]{S^{#1}_{#2}\left(#3\right)} 
\newcommand{\SingVOA}[1]{\mathcal{M}(#1)} 
\newcommand{\KacT}[1]{\mathcal{T}_{#1}} 

\newcommand{\bea}{\begin{eqnarray}}
\newcommand{\eea}{\end{eqnarray}}
\newcommand{\be}{\begin {equation}}
\newcommand{\ee}{\end{equation}}

\newcommand{\W}{\mathcal W}

\newdimen{\Virwidth}

\newcommand{\FF}[1]{\mathcal{F}_{#1}}          
\newcommand{\SingKer}[1]{\mathcal{K}_{#1}}  
\newcommand{\SingIm}[1]{\mathcal{I}_{#1}}   
\newcommand{\SingImP}[1]{M_{#1}}   
\newcommand{\VirIrr}[1]{\mathcal{L}_{#1}}   
\newcommand{\chmap}{\mathrm{ch}}
\newcommand{\ch}[1]{\chmap \bigl[ #1 \bigr]}     
\newcommand{\thetaf}[2]{\theta_{#1}\left(#2\right)}    
\newcommand{\ptheta}[2]{P_{#1}\left(#2\right)}         
\newcommand{\regptheta}[3]{P^{#1}_{#2}\left(#3\right)} 

\newcommand{\scr}[1]{\mathcal{Q}_{#1}}          
\newcommand{\scrs}[2]{\mathcal{Q}_{#1}^{[#2]}}  

\newcommand{\strip}[1]{\mathbb{S}(#1)}          

\begin{document}

\title{On Regularised Quantum Dimensions of the Singlet Vertex Operator Algebra and False Theta Functions}

\author{Thomas Creutzig, Antun Milas and Simon Wood} 

\thanks{T.C. was supported by an NSERC Research Grant (RES0020460).}
  \thanks{ A.M. was  supported by a Simons Foundation grant.}
  \thanks{S.W. was
  supported by an ARC Discovery Early Career Researcher Award (DE140101825)}

 \maketitle

\begin{abstract}

We study a family of non-\(C_2\)-cofinite vertex operator algebras, called the
singlet vertex operator algebras, and connect several important concepts
in the theory of vertex operator algebras, quantum modular forms,  and modular tensor categories. 
More precisely, starting from explicit formulae for characters of modules over the singlet
vertex operator algebra, which can be expressed in terms of false theta functions and
their derivatives, we first deform these characters by using a complex parameter 
$\epsilon$.  We then apply modular transformation properties of regularised
partial theta functions to study asymptotic
behaviour of regularised characters of irreducible modules and compute their regularised quantum dimensions.
We also give a purely geometric description of the 
regularisation parameter as a uniformisation parameter of the fusion variety coming from atypical blocks.
It turns out that the quantum dimensions behave very differently 
depending on the sign of the real part of \(\epsilon\). 
The map from the space of characters equipped with the Verlinde product to the space of regularised
quantum dimensions turns out to be a genuine ring isomorphism for positive
real part of \(\epsilon\)
while for sufficiently negative real part of \(\epsilon\) its surjective image 
gives the fusion ring of a rational vertex operator algebra. 
The category of modules of this rational vertex operator algebra should be viewed as obtained
through the process of a semi-simplification procedure widely used in the theory of quantum groups. 
Interestingly, 
the modular tensor category structure constants of this vertex operator algebra, can be also detected from vector valued quantum modular forms formed by distinguished atypical characters. 

\end{abstract}

\section{Introduction}

Vertex operator algebras and modular objects are intimately connected through the study of their graded dimensions  of modules, also known as characters. 
If the vertex operator algebra is rational, and satisfies a certain cofiniteness condition, this connection is well known, while in the non-rational case the situation is far from being understood. 
Let us briefly recall the rational story. 

\subsection{Quantum dimensions and the Verlinde formula of rational vertex operator algebras }

In the late 80's 
Verlinde conjectured an intriguing relation between the modular properties of characters of a
certain two-dimensional conformal field theory and its fusion ring \cite{V}. Shortly after, Moore and Seiberg explained how this relation arises from the axioms of rational
conformal field theory \cite{MS}. The algebraic axiomatisation of (chiral) two-dimensional conformal field theory is 
given by the representation theory of vertex operator algebras. 
The important question was thus to understand precisely the connection between
modular forms and rational vertex operator algebras on a mathematical level of rigour. 
In the early 90's, Zhu \cite{Z} proved that the space of torus one
point functions, and especially characters of modules, of a rational vertex
operator algebra satisfying the $C_2$-cofiniteness condition, is finite
dimensional and carries an action of the modular group $SL(2; \Z)$.  Finally,
in 2005, Huang \cite{Hu2} proved that the representation category of a rational
$C_2$-cofinite vertex operator algebra (with some additional mild properties)
has the structure of a modular tensor category. 

Equivalence classes of simple
objects in a modular tensor category carry a projective action of the modular
group and it turns out that this categorical projective $SL(2; \Z)$ action
descends to the modular one. An important property of a modular tensor
category is that the Verlinde formula holds (see e.g. the book \cite{T}). 
Let $V$ be a rational vertex operator algebra with simple modules  $M_0=V, M_1, \dots, M_n$, then the characters are the graded (by conformal dimension) trace functions
\begin{displaymath}
\ch{M_i}(\tau):= \tr_{M_i}\left( q^{L_0-\frac{c}{24}}\right),
\end{displaymath}
 with $q=e^{2\pi i\tau}$ and $c$ the central charge of $V$. The modular $S$-transformation defines a $(n+1)\times (n+1)$ matrix $S_{ij}$ via
\begin{displaymath}
\ch{M_i}\left(-\frac{1}{\tau}\right) = \sum_{j=0}^n S_{ij} \ch{M_j}(\tau).
\end{displaymath}
The Verlinde formula (proven in \cite{Hu}) explicitly computes the fusion coefficients in terms of
the $S$ matrix coefficients. On the space of equivalence classes
of modules $[M_i]$, the product
\begin{displaymath}
[M_i] \times [M_j] = \sum_{k=0}^n{N_{ij}}^k  [M_k]
\end{displaymath}
is given by 
\begin{displaymath}
{N_{ij}}^k = \sum_{\ell=0}^n \frac{S_{i\ell}S_{j\ell}\left(S^{-1}\right)_{\ell k}}{S_{0 \ell}}.
\end{displaymath}
This formula implies that for each $r=0,\dots, n$ the map
\begin{displaymath}
[M_i] \mapsto \frac{S_{ir}}{S_{0 r}}
\end{displaymath}
is a ring homomorphism. 
The numbers $\frac{S_{ir}}{S_{0 r}}$ are called generalised quantum dimensions of the simple object $M_i$, while 
for $r=0$ they are just called quantum dimensions.
Let us assume that the vacuum vector is the element of least conformal dimension of all simple modules of $V$, then the corresponding term dominates in the $q\rightarrow 0^+$ limit and hence 
the categorical dimension $\frac{S_{i0}}{S_{0 0}}$
satisfies
\begin{equation} \label{qdim-def}
\frac{S_{i 0}}{S_{0 0}} =
\lim_{\tau\rightarrow 0^+} \frac{\ch{M_i}\left({\tau}\right) }{\ch{M_0}\left({\tau}\right) },
 \end{equation}
where \(\tau\) approaches 0 from the upper half plane along the imaginary axis.
This relation between quantum dimensions and asymptotic behaviour of characters was first noted in \cite{DV}. 
For a thorough discussion of quantum dimensions in rational vertex operator algebras see \cite{Hu2} and \cite{DJX} (see also \cite{M} and \cite{BM}), where a striking relation to indices of subfactors is also pointed out. 

\subsection{Modularity  and non-$C_2$-cofinite vertex operator algebras}

Non-rational vertex operator algebras can further be subdivided into two classes:
$C_2$-cofinite and non-$C_2$-cofinite vertex operator algebras. $C_2$-cofiniteness is a condition that implies
that the vertex operator algebra has only finitely many simple modules, up to isomorphism. Even though there is no general Verlinde-type formula
for $C_2$-cofinite vertex operator algebras,  results of Miyamoto \cite{Miy} guarantee, under the additional assumption that all modules are infinite dimensional, that at least the space of characters, 
supplemented by pseudo characters, is modular invariant. The $(p_+, p_-)$-triplet $W$-algebra is the best
understood family of this type \cite{AdM1,AdM2,FGST1,FGST2,Kau,GK,NT,TW}.
In these cases, irreducible characters and their modular properties have been
computed. They are sums of vector valued modular forms of weight
zero, one and two. So in some sense we get a picture 
that resembles the rational case. For the $(1,p)$-triplet algebra there is
even a proposal for a generalised Verlinde formula indirectly obtained from
the action of $SL(2; \Z)$ on characters \cite{FHST}.

Verlinde-type algebras in the case of non-rational, non $C_2$-cofinite vertex operator
algebras have been studied by the authors and D. Ridout
\cite{CR1,CR2,CR3,CM,RW}. These vertex operator
algebras are expected to have uncountably many isomorphism classes of
modules. The best known example is the rank one Heisenberg vertex operator algebra.
The modular $S$-transformation of Heisenberg characters defines an integral
kernel, the $S$-kernel, in terms of which the Verlinde formula can be expressed.
This is explained in great detail in Section 1.2 of \cite{CR3}. 

One general feature of these irrational theories is that they admit typical
modules (labelled by a continuous parameter) and atypical modules
(parametrised by a discrete set).  When it comes to modular transformation
properties, the $S$-transformation on the character should produce both typical and atypical characters.
So we expect
\begin{equation} \label{cont-S}
\ch{M}(-1/\tau) = \underbrace{ \int_{\Omega} S_{M,\nu} \ch{M_{\nu}}(\tau) d \nu }_{{\rm continuous  \ \ part}} +\underbrace{ \sum_i \alpha_{M,i} \ch{M_i} (\tau)}_{{\rm discrete \ part}},
\end{equation}
where $\ch{M_i}$ are certain atypical characters and $S_{M,\nu}$ is the $S$-kernel defined on some domain $\Omega$ parametrising typicals.
The reader should exercise caution here as the above relation might only exist in a distributional setting (see for instance \cite{CR1}, \cite{CR2}, \cite{RW}); however, see also below. 
In the case of the Heisenberg vertex operator algebras, only the continuous Gaussian integral appears as there are no atypical modules.
This is no surprise as the representation category of the Heisenberg algebra is semi-simple under suitable conditions on modules. Other more complicated examples of non $C_2$-cofinite vertex operator algebras include 
reducible but indecomposable modules. 
In many examples interesting modular-like objects appear in characters.  For
example, mock modular forms appear in the atypical modules of many super
vertex operator algebras due to important work of Kac and Wakimoto, see
e.g. \cite{KW1}. In \cite{AC}, the mock modular properties were related
to the fusion ring of a family of such theories. Meromorphic Jacobi forms of
both positive and negative index also appear as the analytic continuations of
highest-weight module characters of affine vertex operator algebras at admissible,
rational level \cite{KW2}. Here the relation of modular transformations to the
Verlinde algebra is subtle as one has to work with the literal trace
functions, that is, series expansions where the coefficients are dimensions of
weight spaces and not their analytic continuations, see \cite{CR1, CR2}. 

The case we are interested in here concerns certain vertex operator algebras whose atypical characters are essentially determined by certain false theta functions
and their derivatives. A classical {\em partial} theta series is the theta like sum  given by 
\begin{displaymath}
P_{a,b}(q) = \sum_{m=0}^\infty q^{a(m+\frac{b}{2a})^2},\quad a,b \in\mathbb{N}.
\end{displaymath}
A \emph{false} theta function is a difference of two partial
theta functions with fixed $a$. These objects appear in connection to representation theory \cite{BM,CM}, mock modular forms, and as Fourier coefficients of negative index meromorphic Jacobi forms \cite{BCR}.
These appearances are related as the $\SingVOA{p}$ singlet algebras are the coset vertex operator algebras of a family of W-algebras whose atypical
characters can be analytically continued to meromorphic Jacobi forms and their Fourier coefficients are essentially
singlet algebra characters \cite{CRW}.
The definition of false theta functions can also be modified to higher ranks \cite{BM2}. 

\subsection{The $\epsilon$-plane and quantum dimensions of the singlet vertex operator algebra}

In this paper we are primarily concerned with the $\SingVOA{p_+,p_-}$ singlet
algebra, which is defined as the intersection of the kernels
of two screening operators. 
Additionally we will also present results for the $\SingVOA{p}$ singlet algebra studied previously in \cite{CM}.

As previously shown in \cite{CM}, a modular group action on $\SingVOA{p}$ atypical
characters only exists, in the sense of (\ref{cont-S}) being a genuine
equality of functions rather than distributions, if we introduce a regularisation 
parameter $\epsilon$ and view $\ch{M^\epsilon}(\tau)$ as a function of both $\epsilon$
and $\tau$. There are several reasons why regularisation is useful. For example,
it allows us to distinguish each irreducible module by its regularised character. 
Moreover, if  we require $\epsilon \in \mathbb{C} \setminus i \mathbb{R}$, then atypicals transform as
\begin{equation} \label{cont-S1p}
\ch{M^\epsilon}(-1/\tau) = \underbrace{ \int_{\mathbb{R}} S^\epsilon_{M,\nu}
  \ch{\mathcal{F}^\epsilon_{\nu}}(\tau) d \nu }_{{\rm continuous\ part}}
+\underbrace{ G_{M}^\epsilon(\tau)}_{\substack{\rm discrete\\ \rm correction\ part}},
\end{equation}
where $\mathcal{F}_\nu$ are Fock modules and $S^\epsilon_{M,\nu}$ is now the regularised $S$-kernel. 
Our first main result gives such a formula for the characters of atypical
$\SingVOA{p_+,p_-}$ and $\SingVOA{p}$ modules including 
explicit formulae for the correction term expressed in theta functions (see
Proposition \ref{prop:S-trafo} together with equation \eqref{eq:charI+} and Theorem
\ref{S-atypical}).  A similar formula, but without correction term or
regularisation parameter, was found in \cite{RW} by considering characters as
algebraic distributions. The regularisation scheme considered in this article has
the advantage of seeing the correction term which gives rise to the
rich structure discussed below as well as to a parametrisation of the fusion variety discussed below.
In contrast to the $\SingVOA{p}$ vertex operator algebra
this term does {\em not} vanish when $\epsilon$ goes to zero. 

Keeping in mind formula (\ref{qdim-def}), which is valid for rational models,
we {\em define} the regularised quantum dimension of $M$ to be
\begin{displaymath}
\qdim{M^\epsilon}=\lim_{\tau \rightarrow 0^+} \frac{\ch{M^\epsilon}(\tau)}{\ch{V^\epsilon}(\tau)},
\end{displaymath}
where $V$ is the vertex operator algebra; in our case $\SingVOA{p_+,p_-}$ or
\(\SingVOA{p}\).
As was shown in \cite{BM}, (non-regularised) quantum dimensions 
should enjoy nice properties even beyond rational vertex operator algebras.

The aim of this article is to compute a conjectural formula for the fusion
ring of the singlet algebra by means of regularised quantum dimensions.
While computing these formulae, we also obtain many analytic results which are
purely number theoretic in nature, for example, asymptotic properties of
quotients of characters and modular-like transformation 
properties of ``weight'' $\frac{3}{2}$ false theta functions.
What is interesting is that the quantum dimensions behave very differently in the
two disjoint domains separated by the ``wall''\footnote{
Of course, one can also use $i \epsilon$ as a regularisation parameter instead of
$\epsilon$.  This leads to the disjoint domains
$\im(\epsilon)>0$ and $\im(\epsilon)<0$, respectively.}:
\begin{displaymath}
B_\epsilon^q = - 
 \mathrm{min}\left\{ \Big|\frac{m}{\sqrt{2 q}}-\mathrm{Im}\left(\epsilon\right)\Big| \  \Big\vert \ m\in\mathbb Z\setminus q\mathbb Z\ \right\}, 
\end{displaymath}
where $q=p$ or $q=p_-p_+$ for $\SingVOA{p}$ or $\SingVOA{p_+, p_-}$,
respectively (see Figure \ref{fig:wall} for a visualisation of the wall \(B_\epsilon^q\)).

Our findings can be summarised 
as follows.  For $\re(\epsilon)>B_\epsilon^q$, the discrete part in the
$S$-transformation formula vanishes and only the continuous part
contributes to the asymptotic properties of regularised characters. 
Regularised quantum dimensions of irreducibles in this regime are defined everywhere, and are
 given explicitly by (\ref{qdim-cont}). 
We also show that the regularised Verlinde-type formula for irreducible modules 
matches the main result of \cite{RW} (see Theorem \ref{fusion.ring}).  In particular,
$\ch{M^\epsilon}(\tau) \rightarrow \qdim{M^\epsilon}$ is injective (see
Theorem \ref{injective}).  This ring is conjecturally the Grothendieck (or
fusion) ring  of a suitable quotient of the module category of the singlet algebra.  

The region $\re(\epsilon)<B_\epsilon^q$ is more subtle. Here the quantum dimensions are 
strip-wise constant functions of \(\epsilon\).  This sharp difference comes from the fact that 
the correction part now not only contributes to the modular transformation
properties of regularised characters, but actually dominates the continuous
part in the asymptotics.

Regularised quantum dimensions of atypicals in this regime
are given in formula \eqref{atypicals-qdim}. 
Moreover, the fusion ring structure 
leaks through the wall at the drip points \(B_\epsilon^q\cap i\mathbb{R}\),
meaning that at the end of each open strip there is a point on the wall
 where the limit from the left and right agree.
The surjective image of the (conjectural) fusion ring is then isomorphic to the fusion ring of minimal models
(see Theorem \ref{injective} for the precise statement).

\subsection{Semi-simplification and quantum modular forms }
\label{sec:semisimplification}

The results obtained in Proposition \ref{fusion-1p} and Theorem
\ref{injective} are somewhat surprising because the correction  
term appears to capture modular data of a rational VOA, that is, of the
Virasoro minimal models and the $A^{(1)}_1$ WZW models, depending on
\(\epsilon\) for the $\SingVOA{p_+,p_-}$ algebra; 
and of $A^{(1)}_1$ WZW models for the $\SingVOA{p}$ algebra.

We provide two explanations of this phenomenon. One is purely categorical, but
requires that we postulate the existence of an appropriate categorical structure 
on the tensor category of modules.
While the other is a computational approach that reproduces the modular data of
the minimal models and that of the $A_1^{(1)}$ WZW models directly 
from irreducible atypical characters.

Let us first consider  the $\SingVOA{p}$ case in the categorical approach.
Here we expect the category of modules to be   rigid, which we assume for the moment.
The representation category of $\SingVOA{p}$ contains standard modules $\FF{\lambda}$.
As one can see in Section \ref{sec:1,p}, their quantum dimensions in the half-plane where the correction term dominates
is zero. We can regard the standard modules $\FF{\lambda}$ as negligible objects in the 
tensor category and they form an ideal. Thus, it is possible to restrict 
the category of all $\SingVOA{p}$-modules to just the category of atypical blocks, that is, 
to modules with finite composition series in terms of atypicals. We then form a new category 
where  the $\hom$ spaces are moded out by negligible morphisms.  This procedure of moding out negligible morphisms is well-known in the theory of quantum groups (see \cite{AP} for details). 
In particular, in the quotient category the standard modules $\mathcal{F}_{\lambda}$ are isomorphic
to the zero object. This quotient category (also known as  the ``semi-simplification'') is expected 
to be a modular tensor category.

For the $\SingVOA{p_+,p_-}$ vertex operator algebra, things are more subtle. We do not expect the category
of modules to be rigid and therefore there is no notion of a categorical trace. 
Instead, we first mod out the full category by the tensor ideal of minimal models (both \cite{RW} and our paper provide sufficient
evidence for the existence of such an ideal). Then the resulting quotient
category - after  restricted to a full subcategory - is conjecturally braided
and rigid. Assuming that, we proceed as above and after semi-simplification,
we expect to get the $(p_+,p_-)$-minimal models. 
In summary, we conjecture that:
\begin{itemize}
\item[(a)] The quotient category of $\SingVOA{p_+,p_-}$ singlet modules with respect to the tensor ideal of Virasoro minimal models is braided and rigid. Its ``semi-simplified''
category is equivalent to the modular tensor category of $(p_+,p_-)$-Virasoro minimal models. 
\item[(b)]  The category of $\SingVOA{p}$ singlet modules is braided and rigid.  Its ``semi-simplified'' category is equivalent to the modular tensor category of $A_1^{(1)}$ at level $p-2$. 
\end{itemize}

Although elegant,  this categorical approach is largely conjectural. So we
turn to the more concrete computational approach.
The main issue when dealing with singlet
characters is that they do not look anything like finite-dimensional vector
valued modular forms. However, the  problem of infinite dimensionality is 
easily cured by taking the quotient of the space of characters by the subspace
of characters of modules with vanishing quantum dimension.
This leaves us with a finite-dimensional space of distinguished atypical characters.  
In order to extract an $S$-matrix from this space we make use of quantum modular forms ({\em qmf}) introduced in \cite{Za}.
Quantum modular forms are of interest in connection to mock modular forms and invariants of
$3$-manifolds, so it is not a priori clear whether they play an important role
in the representation theory of irrational vertex operator algebras (see \cite{BCR,BM} for some indications that the two subjects are indeed related).
The main property of qmf, at least  for purposes of this paper, is that they
live { both} in the upper and the lower half-plane; and that radial limits from
both sides agree on a distinguished subset of the rationals - the {\em quantum
  set}. It is already known, that one can form a quantum modular form on the upper half-plane by choosing 
a false theta function, while on the lower half 
it is an Eichler integral. Then we prove that  
there is a vector valued quantum modular form of weight $\frac{1}{2}$ that
transforms in exactly the same way as  characters of the  $A_1^{(1)}$ WZW models
at level $p-2$ (see Theorem  \ref{qm-1p}).
Similar constructions for the $\SingVOA{p_+,p_-}$ singlets give rise to a vector 
valued quantum modular form of weight $\frac{3}{2}$ whose $S$-matrix is the one 
of the $(p_+,p_-)$-minimal models (see Theorem  \ref{32}).

\subsection{Fusion varieties}
The fusion rings of the singlet algebras in the ${\rm Re}(\epsilon)>B_\epsilon^q$ regime are infinite-dimensional commutative algebras with uncountably many generators
due to the continua of standard modules. If we restrict ourselves to the
subcategory discussed in Section \ref{sec:semisimplification}, we find that its fusion 
ring can be described as a quotient of a certain polynomial ring (see Proposition \ref{chebyshev}).
As such it can be viewed as the ring of functions on an algebraic variety - the {\em fusion variety}. 
For all singlet vertex operator algebras we explicitly describe this singular curve, which is always of genus zero (see Theorem \ref{fusion-variety}).
From the perspective of fusion varieties, the $\epsilon$-parameter is essentially the uniformisation parameter.
This is another reason why $\epsilon$-regularisation is an important ingredient in studying irrational vertex operator algebras.
 
\subsection{Summary of Results}

This work is structured as follows. In Section 2, the Heisenberg vertex algebra and singlet vertex algebras are introduced. 
In Section 3, characters of modules of these vertex algebras are stated.  
The original work starts with Section 4, where the characters are regularised following and generalising previous ideas \cite{CM}. Their modular properties are then computed in Section 5. These properties depend on the regularisation parameter and explicit formulae are stated in Theorem 8, the main result of this section.
In Section 6 the modular properties are used to determine the quantum dimensions of modules (Propositions 11, 12, 17 and 19) which are then compared to the Verlinde rings (Theorem 20) obtained by different methods in \cite{RW, CM}. The dependence on the regularisation parameter is illustrated in Figure 1.
In Section 7, varieties whose rings of functions are the Verlinde rings of the previous section are studied (Proposition 23 and Theorem 24). Theorem 26 then relates quantum dimensions of the singlet vertex algebras to those of well-known rational vertex algebras. 
 Finally, in Section 8, a vector-valued quantum modular form is associated to each singlet vertex algebra (Theorems 31 and 32).

\section{The Heisenberg and singlet vertex operator algebras}

In this section, we recall well known structures needed to describe the singlet vertex operator algebra.
We use the notation of \cite{FGST1,FGST2} (see also \cite{AdM2}, \cite{TW} for related results).
Let \(p_+,p_-\geq 1\), \(p_+\neq p_-\) and
\(\operatorname{gcd}(p_+,p_-)=1\). In later sections we will occasionally consider
the case when \(p_+,p_-\geq2\) and the case when either \(p_+=1\) or \(p_-=1\) separately. Additionally let
\begin{equation*}
  \alpha_+=\sqrt{\frac{2p_-}{p_+}}, \qquad\alpha_-=-\sqrt{\frac{2p_+}{p_-}},\qquad
  \alpha_0=\alpha_++\alpha_-,\qquad \alpha=p_+\alpha_+=-p_-\alpha_-=\sqrt{2p_+p_-},
\end{equation*}
and
\begin{align*}
  \alpha_{r,s}=\frac{1-r}{2}\alpha_++\frac{1-s}{2}\alpha_-=\alpha_{r+p_+,s+p_-},\qquad
  \alpha_{r,s;n}=\alpha_{r,s+np_-}=\alpha_{r-np_+,s}=\alpha_{r,s}+\frac{n}{2}\alpha\,,
\end{align*}
for \(r,s,n\in\mathbb{Z}\).

Let \(\hat{\mathfrak{h}}\) be the rank 1 extended Heisenberg Lie algebra with generators \(a_n,n\in\mathbb{Z}\), 
satisfying the commutation relations
\begin{align*}
  [a_m,a_n]=m\delta_{m,-n} C\,
\end{align*}
where $C$ is central.
We denote by \(\hat{\mathfrak{h}}_+\) and \(\hat{\mathfrak{h}}_-\) the
commutative Lie algebras generated by Heisenberg  generators labelled by
positive and negative integers respectively.
Let \(\FF{\lambda},\lambda\in\mathbb{C}\) be the standard Fock module over
\(U(\hat{\mathfrak{h}})\) generated by a highest weight vector $v_\lambda$ such that
\begin{align*}
  a_n v_{\lambda} &= \delta_{n,0}\lambda v_\lambda ,\ n\geq0\\
  U(\hat{\mathfrak{h}}_-) v_\lambda &=\FF{\lambda}\,
\end{align*}
where $C$ acts as $1$.

The Fock module \(\FF{0}\) carries the structure of a vertex algebra, which is
generated by the field
\begin{align*}
  Y(a_{-1}{\bf 1},z)=a(z)=\sum_{n \in \mathbb{Z}}a_n z^{-n-1},
\end{align*}
and where the vacuum vector is ${\bf 1}=v_0$. This vertex algebra becomes a
vertex operator algebra by choosing the conformal vector to be
\begin{align*}
  \omega= \frac{1}{2}(a_{-1}^2+\alpha_0a_{-2}){\bf 1}.
\end{align*}
The Fourier modes $L(n)$ of \(\omega\) then span the Virasoro algebra
with central charge
\begin{align*}
  c=1-3\alpha_0^2=1-6\frac{(p_+-p_-)^2}{p_+p_-}\,.
\end{align*}
This choice of conformal structure determines the conformal dimension of
the highest weight vectors \(v_\lambda\in\FF{\lambda}\) to be
\begin{displaymath}
\Delta_\lambda=\tfrac{1}{2}\lambda(\lambda-\alpha_0)=\tfrac12(\lambda-\tfrac{\alpha_0}{2})^2-\tfrac{\alpha_0}{8}.
\end{displaymath}

The minimal model vertex operator algebras at
\(c=1-6\frac{(p_+-p_-)^2}{p_+p_-}, p_+,p_-\geq2\) can be realised as subquotients of Fock
modules \cite{Fel,IK}. The representation theory of these vertex operator algebras
is completely reducible and the distinct isomorphism classes of irreducible
representations are labelled by the set 
\begin{displaymath}
  \KacT{p_+, p_-}=\left\{(r,s)\left|\ 1 \leq r < p_+ ,1 \leq s < p_- , s p_+ > r p_-\right.\right\}.
\end{displaymath}
We denote these irreducible representations by \(\VirIrr{r,s}\) and their
conformal dimension is $\Delta_{r,s}=\frac{(p_+ r-p_- s)^2-(p_+ - p_-)^2}{4 p_+  p_-}$. 
In this parametrisation the vacuum module is given by \(\VirIrr{1,1}\). Note
that \(\Delta_{r,s}=\Delta_{p_+-r,p_--s}\) which is why the irreducible minimal
model representations are labelled by \(\KacT{p_+,p_-}\) rather than the set \(\{(r,s)|1\leq
r<p_+, 1\leq s <p_-\}\).

\subsection{The singlet vertex operator algebra}

The Feigin-Fuchs
modules $\FF{r,s;n}:= \FF{\alpha_{r,s;n}}$ organise into so called Felder
complexes \cite{Fel} by
means of the two commuting screening operators \(\scr{+}\) and \(\scr{-}\).
For \(1\leq r <p_+\), \(1\leq s\leq p_-\), \(n\in\mathbb{Z}\)
\begin{equation}
  \begin{tikzpicture} 
    \path (0,0) node (1) {\(\cdots\)} 
    ++(3,0) node (2) {\(\FF{p_+-r,s;n-1}\)} edge [<-] node [above] {\(\scrs{+}{r}\)} (1)
    ++(3,0) node (3) {\(\FF{r,s;n}\)} edge [<-] node [above] {\(\scrs{+}{p_+-r}\)} (2) 
    ++(3,0) node (4) {\(\FF{p_+-r,s;n+1}\)} edge [<-] node [above] {\(\scrs{+}{r}\)} (3)
    ++(3,0) node (5) {\(\cdots\)} edge [<-] node [above] {\(\scrs{+}{p_+-r}\)} (4);
\end{tikzpicture}
\end{equation}
while for \(1\leq r \leq p_+\), \(1\leq s< p_-\), \(n\in\mathbb{Z}\)
\begin{equation}
 \begin{tikzpicture}
    \path (0,-2) node (6) {\(\cdots\)} 
    ++(3,0) node (7) {\(\FF{r,p_--s;n+1}\)} edge [<-] node  [above] {\(\scrs{-}{s}\)} (6) 
    ++(3,0) node (8) {\(\FF{r,s;n}\)} edge [<-] node  [above] {\(\scrs{-}{p_--s}\)} (7) 
    ++(3,0) node (9) {\(\FF{r,p_--s;n-1}\)} edge [<-] node [above] {\(\scrs{-}{s}\)} (8) 
    ++(3,0) node (10) {\(\cdots\,.\)} edge [<-] node  [above] {\(\scrs{-}{p_--s}\)} (9);
  \end{tikzpicture}
\end{equation}
We omit explicit formulae for the ``powers'' $\scrs{\pm}{i}$ as they involve contour integrals (see \cite{TW} and references therein for more details).
These complexes are exact everywhere but \(\FF{r,s;0}\) if \(r<p_+\) and \(s<p_-\). If
either \(r=p_+\) or \(s=p_-\), then they are exact everywhere. 

These two commuting screening operators conveniently allow one to define three
vertex operator subalgebras of \(\FF{0}=\FF{1,1;0}\).
\begin{definition} For \(p_+,p_-\geq2\), let
  \begin{displaymath}
    \begin{split}
      \SingVOA{p_+,p_-}^+&=\ker \scr{+}: \FF{1,1;0}\rightarrow \FF{p_+-1,1;1},\quad
      p_+\geq2,\\
      \SingVOA{p_+,p_-}^-&=\ker \scr{-}: \FF{1,1;0}\rightarrow
      \FF{1,p_--1;-1},\quad p_-\geq2,\\
      \SingVOA{p_+,p_-}&=\SingVOA{p_+,p_-}^+\cap \SingVOA{p_+,p_-}^-,\quad p_+,p_-\geq2,
    \end{split}
  \end{displaymath}
  while for \(p_-=1\) or \(p_+=1\), let
  \begin{displaymath}
    \begin{split}
      \SingVOA{p_+}&=\ker \scr{+}: \FF{1,1;0}\rightarrow \FF{p_+-1,1;1}, p_+\geq2,\\
      \SingVOA{p_-}&=\ker \scr{-}: \FF{1,1;0}\rightarrow \FF{1,p_--1;1}, p_-\geq2,
    \end{split}
  \end{displaymath}
  respectively.
\end{definition}
The vertex operator algebra $\SingVOA{p_+,p_+}$ is called the $(p_+,p_-)$-singlet
algebra while \(\SingVOA{p_+}\) and \(\SingVOA{p_-}\) are the \((p_+,1)\)-
and \((1,p_-)\)-singlet algebras, respectively. 
It was proven in \cite{AdM2} for $p_-=2$ and for general $p_-$ in \cite{TW} that $\SingVOA{p_+,p_-}$ is a $\W$-algebra of the singlet type, in the sense 
that is strongly generated by the Virasoro field and another primary field.
The vertex operator algebras $\SingVOA{p_+,p_-}^+$ and $ \SingVOA{p_+,p_-}^-$ have not been studied in 
the literature perhaps because they are not $\mathcal{W}$-algebras according
to standard definitions unless either \(p_+\) or \(p_-\) is 1. If \(p_+=1\)
then \(\SingVOA{p_-}=\SingVOA{1,p_-}^-\) and if \(p_-=1\) then
\(\SingVOA{p_+}=\SingVOA{p_+,1}^+\).

The minimal model vertex operator algebra is given by the (equivalent)
cohomologies of either
\(\scr{+}\) or \(\scr{-}\) at \(\FF{1,1;0}\), while the remaining minimal
model representations are given by the cohomologies at \(\FF{r,s:0}\).

\subsection{Representation theory}

Since \(\SingVOA{p_+,p_-}^\pm\) and \(\SingVOA{p_+,p_-}\) are vertex
operator subalgebras of \(\FF{0}\), the Fock modules \(\FF{\lambda}\) are also modules
over these subalgebras. When considered as modules over
\(\SingVOA{p_+,p_-}\) or \(\SingVOA{p_\pm}\),
we call the the Fock modules \emph{standard modules}. Additionally we define
the \emph{typical modules} to be the standard
modules that are simple over the vertex operator algebra being
considered. All indecomposable modules that are not typical are called \emph{atypical modules}.
\begin{definition}
  For \(n\in\mathbb{Z}\) let \(\SingKer{r,s;n}^\pm\), \(\SingIm{r,s;n}^\pm\),
  \(\SingKer{r,s;n}\) and \(\SingIm{r,s;n}\) be the following subspaces of the
  \(\FF{r,s;n}\):
  \begin{displaymath}
    \begin{split}
      \SingKer{r,s;n}^+&=\ker \scrs{+}{r}: \FF{r,s;n}\rightarrow
      \FF{p_+-r,s;n+1},\\
      \SingKer{r,s;n}^-&=\ker \scrs{-}{s}: \FF{r,s;n}\rightarrow
      \FF{r,p_--s;n-1},\\
      \SingKer{r,s;n}&=\SingKer{r,s;n}^+\cap \SingKer{r,s;n}^-,
    \end{split}\quad
    \begin{split}
      p_+\geq2, 1\leq r< p_+,1\leq s\leq p_-,\\
      p_-\geq2, 1\leq r\leq p_+,1\leq s< p_-,\\
      p_+,p_-\geq2, 1\leq r< p_+,1\leq s< p_-.
    \end{split}
  \end{displaymath}
  \begin{displaymath}
    \begin{split}
      \SingIm{r,s;n}^+&=\img \scrs{+}{p_+-r}: \FF{p_+-r,s;n-1}\rightarrow \FF{r,s;n},\\
      \SingIm{r,s;n}^-&=\img \scrs{-}{p_--s}: \FF{r,p_--s;n+1}\rightarrow \FF{r,s;n},\\
      \SingIm{r,s;n}&=\SingIm{r,s;n}^+\cap \SingIm{r,s;n}^-,
    \end{split}\quad
    \begin{split}
      p_+\geq2, 1\leq r<p_+,1\leq s\leq p_-,\\
      p_-\geq2, 1\leq r\leq p_+,1\leq s< p_-,\\
      p_+,p_-\geq2, 1\leq r<p_+,1\leq s< p_-.
    \end{split}
  \end{displaymath}
\end{definition}
Note that the singlet vacuum is \(\SingKer{1,1;0}=\SingVOA{p_+,p_-}\) and that \(\SingKer{1,1;0}^\pm=\SingVOA{p_+,p_-}^\pm\).
The $\SingKer{r,s;n}$ and the $\SingIm{r,s;n}^{\pm}$ are modules over
$\SingVOA{p_+,p_-}^\pm$ and since $\SingVOA{p_+,p_-}=\SingVOA{p_+,p_-}^+ \cap \SingVOA{p_+,p_-}^-$, both $\SingKer{r,s;n}$ and
$\SingIm{r,s;n}$ are $\SingVOA{p_+,p_-}$-modules. 
For notational simplicity it turns out to be convenient to extend the range of
the \(r,s\) label in \(\SingIm{r,s;n}\) to include the border cases when
\(r=p_+\) or \(s=p_-\), so we define
\begin{displaymath}
  \begin{split}
    \SingIm{p_+,s;n}=\SingIm{p_+,s;n}^-,\qquad
    \SingIm{r,p_-;n}=\SingIm{r,p_-;n}^+,\qquad
    \SingIm{p_+,p_-;n}=\FF{p_+,p_-;n},
  \end{split}
\end{displaymath}
where \(1\leq r<p_+\), \(1\leq s<p_-\) and \(n\in\mathbb{Z}\).

In \cite{TW} (cf. also \cite{AdM2}), Zhu's associative algebra of $\SingVOA{p_+,p_-}$ was computed. 
By using its structure one can prove that every irreducible module appears as subquotient 
of a Fock representation. This implies the following theorem.
\begin{thm} \label{class-irrep} Let \(p_+,p_-\geq2\). A simple ($\mathbb{Z}_{\geq 0}$-gradable) $\SingVOA{p_+,p_-}$-module 
is isomorphic to one of the following:
\begin{itemize}
\item[(a)] \emph{Typical}\footnote{The nomenclature ``standard'' and ``typical module'' are both commonly used in the representation theory of logarithmic vertex algebras and have their origin in associative algebras and Lie super algebras, respectively. 
}: Standard modules $\FF{\lambda}$, where $\lambda
  \notin \frac{1}{\sqrt{2p_+ p_-}} \mathbb{Z}$ or $\FF{p_+,p_-;n}, n\in\mathbb{Z}$.
\item[(b)] \emph{Atypical submodules of standard modules}:
    $\mathcal{I}_{r,s;n}$,\\ where $1 \leq r \leq p_+$ and $1 \leq s \leq p_-$ and $n \in \mathbb{Z}$.
\item[(c)] \emph{Atypical subquotients of standard modules}: irreducible Virasoro modules $\mathcal{L}_{r,s}$,\\ where \((r,s)\in\KacT{p_+,p_-}\).
\end{itemize}
\end{thm}

Recall that the Felder complexes are exact away from \(n=0\), thus
\begin{align*}
  \SingKer{r,s;n}^\mu\cong \SingIm{r,s;n}^\mu,\ \mu=\pm,\emptyset,\ n\neq0.
\end{align*}
However, at \(n=0\) when \(1\leq r<p_+,1\leq s<p_-\), the modules
\(\SingKer{r,s;0}, \SingIm{r,s;0}\) satisfy the following exact sequences
\begin{equation}\label{eq:scrkerseq}
  \begin{tikzpicture}
    \path (0,0) node (1) {\(0\)} 
    ++(2,0) node (2) {\(\SingIm{r,s;0}\)} edge [<-]  (1)
    ++(2,0) node (3) {\(\SingKer{r,s;0}\)} edge [<-]  (2)
    ++(2,0) node (4) {\(\mathcal{L}_{r,s}\)} edge [<-]  (3)
    ++(2,0) node (5) {\(0\)\,.} edge [<-]  (4);
  \end{tikzpicture}
\end{equation}

If we consider the \(\scr{+}\) Felder complex, but restrict to
\(\SingIm{r,s;n}^-\subset \FF{r,s;n}\), then because the screening operators \(\scr{+}\)
and \(\scr{-}\) commute, we obtain the complex
\begin{align}\nonumber
  \begin{tikzpicture}
    \path (0,0) node (1) {\(\cdots\)} 
    ++(3,0) node (2) {\(\SingIm{p_+-r,s;n-1}^-\)} edge [<-] node [above] {\(\scrs{+}{r}\)} (1) 
    ++(3,0) node (3) {\(\SingIm{r,s;n}^-\)} edge [<-] node [above] {\(\scrs{+}{p_+-r}\)} (2) 
    ++(3,0) node (4) {\(\SingIm{p_+-r,s;n+1}^-\)} edge [<-] node [above] {\(\scrs{+}{r}\)} (3)
    ++(3,0) node (5) {\(\cdots\)} edge [<-] node [above] {\(\scrs{+}{p_+-r}\)} (4);
  \end{tikzpicture}
\end{align}
which happens to be exact for all \(n\).\footnote{The exactness can be seen be
  comparing socle sequence decompositions.}
Similarly for the \(\scr{-}\) Felder complex we
have the exact sequence
\begin{align*}
  \begin{tikzpicture}
    \path (0,0) node (6) {\(\cdots\)} 
    ++(3,0) node (7) {\(\SingIm{r,p_--s;n+1}^+\)} edge [<-] node  [above] {\(\scrs{-}{s}\)} (6) 
    ++(3,0) node (8) {\(\SingIm{r,s;n}^+\)} edge [<-] node  [above] {\(\scrs{-}{p_--s}\)} (7) 
    ++(3,0) node (9) {\(\SingIm{r,p_--s;n-1}^+\)} edge [<-] node [above] {\(\scrs{-}{s}\)} (8) 
    ++(3,0) node (10) {\(\cdots\,,\)} edge [<-] node  [above] {\(\scrs{-}{p_--s}\)} (9);
  \end{tikzpicture}
\end{align*}

\begin{remark}  The structure of Felder's complex is somewhat simpler if either
   \(p_+\) or \(p_-\) is 1. If \(p_+=1\) then it is given by 
  \begin{displaymath}
    \begin{tikzpicture}
      \path (0,-2) node (6) {\(\cdots\)} 
      ++(3,0) node (7) {\(\FF{r,s}\)} edge [<-] node  [above] {\(\scrs{-}{s}\)} (6) 
      ++(3,0) node (8) {\(\FF{r+1,p_--s}\)} edge [<-] node  [above] {\(\scrs{-}{p_--s}\)} (7) 
      ++(3,0) node (9) {\(\FF{r+2,s}\)} edge [<-] node [above] {\(\scrs{-}{s}\)} (8) 
      ++(3,0) node (10) {\(\cdots\,,\)} edge [<-] node  [above] {\(\scrs{-}{p_--s}\)} (9);
    \end{tikzpicture}
  \end{displaymath}
  and is exact everywhere. Irreducible atypical $\SingVOA{p_-}$ modules are obtained in the range $1 \leq s \leq p_--1$:
  \begin{displaymath}
    \SingImP{r,s}=\ker(\scrs{-}{s}: \FF{r,s} \rightarrow \FF{r+1,p-s})=\img(\scrs{-}{p-s}: \FF{r-1,s} \rightarrow \FF{r,s}).
\end{displaymath}
A classification theorem for $\SingVOA{p_-}$-modules analogous to Theorem \ref{class-irrep} can be found in \cite{CM}.
\end{remark}

\section{Characters of singlet algebra modules}

In this section we will use the Felder complexes to derive character formulae
for atypical modules in terms of the characters of standard modules. These
character formulae will then offer a simple means of regularising the
characters of atypical modules in later sections.

The partial theta function \(\ptheta{a,b}{u,\tau}\) is the power series
\begin{align*}
  \ptheta{a,b}{u,\tau}=\sum_{k\geq0}z^{k+\tfrac{b}{2a}}q^{a(k+\tfrac{b}{2a})^2},\ q=e^{2\pi i\tau},z=e^{2\pi i u}\,,
\end{align*}
where \(a,b\in\mathbb{N},\tau\in\mathbb{H},u\in\mathbb{C}\), while the theta function is given by
\begin{align*}
  \thetaf{a,b}{u,\tau}=\sum_{k\in\mathbb{Z}}z^{k+\tfrac{b}{2a}}q^{a(k+\tfrac{b}{2a})^2},\ q=e^{2\pi i\tau},z=e^{2\pi i u}\,.
\end{align*}
Furthermore, we define
\begin{displaymath}
  \thetaf{a,b}{\tau}=\thetaf{a,b}{0,\tau}, \qquad \ptheta{a,b}{\tau}=\ptheta{a,b}{0,\tau},
\end{displaymath}
and
\begin{displaymath}
  \thetaf{a,b}{\tau}^\prime=z\partial_z \thetaf{a,b}{u,\tau}|_{u=0}, \qquad
  \ptheta{a,b}{\tau}^\prime=z\partial_z \ptheta{a,b}{u,\tau}|_{u=0}.
\end{displaymath}
The theta functions satisfy the well-known modular transformation formula
\begin{equation} \label{mod-theta}
  \thetaf{a,b}{\frac{u}{\tau}, -\frac{1}{\tau}} =  
  \sqrt{\frac{-i\tau}{2a}} {e^{\pi i \frac{u^2}{2a\tau}}} 
  \sum_{c=0}^{2a-1} e^{\frac{-2\pi i b c}{2a}}\thetaf{a,c}{u, \tau},
\end{equation}
while the modular transformation properties of the partial theta functions $\ptheta{a,b}{u,\tau}$ 
are much more delicate and involve the $\epsilon$-regularisation
procedure introduced in \cite{CM}. 

The character of the standard module \(\FF\lambda\) is given by
\begin{displaymath}
  \ch{\FF\lambda}(u,\tau)=\tr_{\mathcal{F}_\lambda} q^{L(0)-c/24}=
  \frac{q^{\Delta_\lambda-\tfrac{c}{24}}}{(q)_\infty}
  =\frac{q^{(\lambda-\frac{\alpha_0}{2})^2/2}}{\eta(q)}\,.
\end{displaymath}
Note that both \(\FF\lambda\) and \(\FF{\alpha_0-\lambda}\) have identical
characters. We will distinguish these two standard modules by introducing a deformation
parameter \(\epsilon\) in Section \ref{sec:regchars}.

The Virasoro minimal model representations are given by the cohomologies of the Felder complexes
and therefore, by the Euler-Poincar\'e principle, their characters are the alternating sums of the characters of all the entries of the complexes
\begin{align*}
  \ch{\VirIrr{r,s}}&=\sum_{k\in\mathbb{Z}}\ch{\FF{r,s;2k}}-\ch{\FF{p_+-r,s;2k+1}}
  =\sum_{k\in\mathbb{Z}}\ch{\FF{r,s;2k}}-\ch{\FF{r,p_-s;2k+1}}\\
  &=\frac{\theta_{p_+p_-,-rp_-+sp_+}(\tau)-\theta_{p_+p_-,rp_-+sp_+}(\tau)}{\eta(\tau)}
  =\frac{\theta_{p_+p_-,rp_--sp_+}(\tau)-\theta_{p_+p_-,rp_-+sp_+}(\tau)}{\eta(\tau)},
\end{align*}
where \((r,s)\in\KacT{p_+,p_-}\).

By choosing resolutions for the \(\SingIm{r,s;n}^+\) and co-resolutions for the
\(\SingIm{r,s;n}^-\) in the 
Felder complexes above, the Euler-Poincar\'e principle yields the following character formulae
\begin{align*}
  \ch{\SingIm{r,s:n}^+}&=(1-\delta_{s,p_-})\delta_{n\geq1}\left(\delta_{n,\text{even}}\ch{\VirIrr{r,s}}
    -\delta_{n,\text{odd}}\ch{\VirIrr{r,p_--s}}\right)\\
  &\quad +\sum_{k\geq0} \ch{\FF{p_+-r,s:n-2k-1}}-\ch{\FF{r,s:n-2k-2}}\\
  &=(1-\delta_{s,p_-})\delta_{n\geq1}\frac{1}{\eta(q)}
  \left(\delta_{n,\text{even}}(\theta_{p_+p_-,-rp_-+sp_+}(\tau)
    -\theta_{p_+p_-,rp_-+sp_+}(\tau))\right.\\
  &\quad \left.-\delta_{n,\text{odd}}(\theta_{p_+p_-,rp_-+sp_+-p_+p_-}(\tau)
    -\theta_{p_+p_-,-rp_-+sp_++p_+p_-}(\tau))\right)\\
  &\quad+
  \frac{\ptheta{p_+p_-,(2-n)p_+p_--rp_--sp_+}{\tau}-\ptheta{p_+p_-,(2-n)p_+p_-+rp_--sp_+}{\tau}}{\eta(q)},
\end{align*}
where \(1\leq r<p_+,1\leq s\leq p_-\),
\begin{align*}
  \ch{\SingIm{r,s:n}^-}&=(1-\delta_{r,p_+})\delta_{n\geq0}\left(\delta_{n,\text{odd}}\ch{\VirIrr{r,p_--s}}
    -\delta_{n,\text{even}}\ch{\VirIrr{r,s}}\right)\\
  &\quad +\sum_{k\geq0} \ch{\FF{r,s:n-2k}}-\ch{\FF{r,p_--s:n-2k-1}}\\
  &=(1-\delta_{r,p_+})\delta_{n\geq0}\frac{1}{\eta(q)}
  \left(\delta_{n,\text{odd}}(\theta_{p_+p_-,rp_-+sp_+-p_+p_-}(0,\tau)
    -\theta_{p_+p_-,-rp_-+sp_++p_+p_-}(0,\tau)\right.\\
    &\quad \left.-\delta_{n,\text{even}}(\theta_{p_+p_-,-rp_-+sp_+}(0,\tau)
    -\theta_{p_+p_-,rp_-+sp_+}(0,\tau)))\right)\\
  &\quad+
  \frac{\ptheta{p_+p_-,-np_+p_-+rp_--sp_+}{0,\tau}-\ptheta{p_+p_-,-np_+p_-+rp_-+sp_+}{0,\tau}}{\eta(q)},\\
\end{align*}
where \(1\leq r\leq p_+,1\leq s< p_-\) and
\begin{displaymath}
  \delta_{n\geq k}=
  \begin{cases}
    1&n\geq k\\
    0& n<k
  \end{cases},\quad k=0,1.
\end{displaymath}
These formulae are structurally similar to the characters of modules over the
$\SingVOA{p}$ singlet algebra studied in \cite{CM}, that is, they are expressible 
as differences of partial thetas divided by the Dedekind \(\eta\)-function. 

For the \(\SingIm{r,s;n}\) we again apply the Euler-Poincar\'e principle to
obtain the characters in terms of \(\SingIm{r,s;n}^\pm\) characters.
\begin{align*}
  \ch{\SingIm{r,s;n}}&=\sum_{k\geq0}\ch{\SingIm{r,s;n-2k}^+}-\ch{\SingIm{r,p_--s;n-2k-1}^+}\\
  &=\sum_{k\geq0}\ch{\SingIm{p_+-r,s;n-2k-1}^-}-\ch{\SingIm{r,s;n-2k-2}^-}\\
&=n\cdot\delta_{n\geq0}\left(\delta_{n,\text{even}}\ch{\VirIrr{r,s}}
  -\delta_{n,\text{odd}}\ch{\VirIrr{r,p_--s}}\right)\\
 &\quad +\sum_{k\geq0}(k+1)\left(
   \ch{\FF{p_+-r,s:n-2k-1}}+\ch{\FF{r,p_--s:n-2k-3}}\right.\\
 &\quad\phantom{+\sum_{k\geq0}(k+1)}\left.
   -\ch{\FF{r,s:n-2k-2}}-\ch{\FF{p_+-r,p_--s:n-2k-2}}\right)\\
 &=\delta_{n\geq1}\frac{n}{\eta(q)}
  \left(\delta_{n,\text{even}}(\theta_{p_+p_-,-rp_-+sp_+}(0,\tau)
    -\theta_{p_+p_-,rp_-+sp_+}(0,\tau))\right.\\
    &\quad \left.-\delta_{n,\text{odd}}(\theta_{p_+p_-,rp_-+sp_+-p_+p_-}(0,\tau)
    -\theta_{p_+p_-,-rp_-+sp_++p_+p_-}(0,\tau))\right)\\
  &\quad\frac{1}{\eta(\tau)} \sum_{k\geq0}(k+1)\left( q^{p_-p_+\left(k+\frac{(2-n)p_+p_- -p_+ s-p_- r}{2 p_+ p_-}\right)^2}+ 
     q^{p_- p_+\left(k+\frac{(2-n)p_+p_- + p_+s+p_- r}{2 p_+ p_-}\right)^2} \right.\\
 &\quad \phantom{\frac{1}{\eta(\tau)} \sum_{k\geq0}(k+1)}\quad 
 \left. - q^{p_- p_+\left(k+\frac{(2-n)p_+p_- -p_+ s+p_- r}{2 p_+ p_-}\right)^2}
     - q^{p_- p_+\left(k+\frac{(2-n)p_+p_- +p_+ s-p_- r}{2 p_+ p_-}\right)^2}
   \right),
\end{align*}
where \(1\leq r< p_+,1\leq s< p_-\).
By the identity
\begin{align*}
  \sum_{k\geq 0}(k+1)q^{a(k+\tfrac{b}{2a})^2}=(1-\tfrac{b}{2a})\ptheta{a,b}{0,\tau}+\ptheta{a,b}{0,\tau}^\prime\,,
\end{align*}
we obtain the following character formulae in terms
of partial theta functions:
\begin{align*}
  &\ch{\SingIm{r,s;n}}=\delta_{n\geq1}\frac{n}{\eta(q)}
  \left(\delta_{n,\text{even}}(\theta_{p_+p_-,-rp_-+sp_+}(0,\tau)
    -\theta_{p_+p_-,rp_-+sp_+}(0,\tau))\right.\\
    &\quad \left.-\delta_{n,\text{odd}}(\theta_{p_+p_-,rp_-+sp_+-p_+p_-}(0,\tau)
    -\theta_{p_+p_-,-rp_-+sp_++p_+p_-}(0,\tau))\right)\\
 &+\frac{1}{\eta(q)}
  \left[\frac{1}{2p_+p_-}\left(
      (np_+p_-+rp_-+sp_+)\ptheta{p_+p_-,(2-n)p_+p_--rp_--sp_+}{0,\tau}\right.\right.\\ &\phantom{\frac{1}{\eta{q}}\frac{1}{2p_+p_-}\ }
  +(np_+p_--rp_--sp_+)\ptheta{p_+p_-,(2-n)p_+p_-+rp_-+sp_+}{0,\tau}\\ &\phantom{\frac{1}{\eta{q}}\frac{1}{2p_+p_-}\ }
  -(np_+p_--rp_-+sp_+)\ptheta{p_+p_-,(2-n)p_+p_-+rp_--sp_+}{0,\tau}\\ &\phantom{\frac{1}{\eta{q}}\frac{1}{2p_+p_-}\ }
  \left.-(np_+p_-+rp_--sp_+)\ptheta{p_+p_-,(2-n)p_+p_--rp_-+sp_+}{0,\tau}\right)\\ &\phantom{\frac{1}{\eta{q}}\ }
 + \ptheta{p_+p_-,(2-n)p_+p_--rp_--sp_+}{0,\tau}^\prime+\ptheta{p_+p_-,(2-n)p_+p_-+rp_-+sp_+}{0,\tau}^\prime\\ &\phantom{\frac{1}{\eta{q}}\ }
  -\ptheta{p_+p_-,(2-n)p_+p_-+rp_--sp_+}{0,\tau}^\prime-\ptheta{p_+p_-,(2-n)p_+p_--rp_-+sp_+}{0,\tau}^\prime\bigg]
\end{align*}

Finally, the character of the \(\SingKer{r,s;0}\) follows from the exact sequence \eqref{eq:scrkerseq}
\begin{align*}
  \ch{\SingKer{r,s;0}}=&\ch{\VirIrr{r,s}}+\ch{\SingIm{r,s;0}}\,.
\end{align*}
In \cite{RW}, character formulae were also derived using Felder complexes, however, with different choices on when to use resolutions or co-resolutions.

\section{Regularised characters of singlet algebra modules}
\label{sec:regchars}

In this section we regularise the characters of irreducible
$\SingVOA{p_+,p_-}$ modules following the methods in \cite{CM}. 
Let $ \epsilon \in \mathbb{C}$.
We define the regularised typical characters to be
\begin{displaymath}
  \ch{\FF{\lambda}^\epsilon} := e^{2\pi \epsilon\left(\lambda-\alpha_0/2\right)}
  \ch{\FF{\lambda}}=
  e^{2\pi \epsilon\left(\lambda-\alpha_0/2\right)} \frac{q^{(\lambda-\alpha_0/2)^2/2}}{\eta(\tau)}.
\end{displaymath}
The regularised characters of atypical modules will then be defined using the
resolutions of previous sections where
unregularised characters of standard modules are replaced by regularised
characters of standard modules.
To more easily give the regularised characters of atypical modules,
we introduce $\epsilon$-regularised partial theta functions 
\begin{equation}
  \regptheta{\epsilon}{a, b}{u , \tau} =\sum_{k\geq0}e^{2\pi\epsilon\left(\tfrac{b}{2a}+k\right)}z^{\tfrac{b}{2a}+k}q^{a\left(\tfrac{b}{2a}+k\right)^2}
\end{equation}
as well as the mixed false theta functions
\begin{displaymath}
F_{b, c}^\epsilon(\tau) := \frac{1}{\eta(\tau)}
\left(P^{-\sqrt{2p_+p_-}\epsilon}_{p_+p_-, b-c}(0, \tau) -
  P^{-\sqrt{2p_+p_-}\epsilon}_{p_+p_-, b+c}(0, \tau)\right),\quad
b,c\in\mathbb{Z}, c\neq0. 
\end{displaymath}
By replacing the characters of typical modules by regularised characters in the
character formulae of the previous section,
we get the following formulae for regularised atypical characters:
\begin{equation}\label{eq:charI+}
\begin{split}
\ch{\SingIm{r,s;n}^{+, \epsilon}}&= (1-\delta_{s,p_-})\delta_{n\geq1}\left(\delta_{n,\text{even}}\ch{\VirIrr{r,s}}
  -\delta_{n,\text{odd}}\ch{\VirIrr{r,p_--s}}\right)\\
&\quad+\sum_{k\geq0}\ch{\FF{p_+-r,s;n-2k-1}^\epsilon}-\ch{\FF{r,s;n-2k-2}^\epsilon}\\
&= (1-\delta_{s,p_-})\delta_{n\geq1}\left(\delta_{n,\text{even}}\ch{\VirIrr{r,s}}
  -\delta_{n,\text{odd}}\ch{\VirIrr{r,p_--s}}\right)\\
  &\quad+\frac{\regptheta{-\sqrt{2p_-p_+}\epsilon}{p_+p_-,(2-n)p_+p_--rp_--sp_+}{\tau}-\regptheta{-\sqrt{2p_-p_+}\epsilon}{p_+p_-,(2-n)p_+p_-+rp_--sp_+}{\tau}}{\eta(q)}\\
&= (1-\delta_{s,p_-})\delta_{n\geq1}\left(\delta_{n,\text{even}}\ch{\VirIrr{r,s}}
  -\delta_{n,\text{odd}}\ch{\VirIrr{r,p_--s}}\right)
+F_{(2-n)p_+p_--sp_+, rp_-}^{\epsilon}(\tau),     
\end{split}
\end{equation}
where \(1\leq r<p_+,1\leq s\leq p_-\).
\begin{equation}
\begin{split}
\ch{\SingIm{r,s;n}^{-, \epsilon}}&= (1-\delta_{r,p_+})\delta_{n\geq0}\left(\delta_{n,\text{odd}}\ch{\VirIrr{r,p_--s}}
   -\delta_{n,\text{even}}\ch{\VirIrr{r,s}}\right)\\
&\quad+\sum_{k\geq0}\ch{\FF{r,s;n-2k}^\epsilon}-\ch{\FF{r,p_--s;n-2k-1}^\epsilon}\\
&=(1-\delta_{r,p_+})\delta_{n\geq0}\left(\delta_{n,\text{odd}}\ch{\VirIrr{r,p_--s}}
   -\delta_{n,\text{even}}\ch{\VirIrr{r,s}}\right)\\
&\quad+ \frac{\regptheta{-\sqrt{2p_-p_+}\epsilon}{p_+p_-,-np_+p_-+rp_--sp_+}{\tau}-\regptheta{-\sqrt{2p_-p_+}\epsilon}{p_+p_-,-np_+p_-+rp_-+sp_+}{\tau}}{\eta(q)}\\
&= (1-\delta_{r,p_+})\delta_{n\geq0}\left(\delta_{n,\text{odd}}\ch{\VirIrr{r,p_--s}}
   -\delta_{n,\text{even}}\ch{\VirIrr{r,s}}\right)
+F_{ -np_+p_-+rp_-, sp_+}^{\epsilon}(\tau),         
\end{split}
\end{equation}
where \(1\leq r\leq p_+,1\leq s<p_-\).
\begin{equation}
\begin{split}
  &\ch{\SingIm{r,s;n}^{\epsilon}}=n\cdot\delta_{n\geq0}\left(\delta_{n,\text{even}}\ch{\VirIrr{r,s}}
    -\delta_{n,\text{odd}}\ch{\VirIrr{r,p_--s}}\right)\\
  &\qquad+\frac{1}{4}\sum_{\nu\in\{\pm 1\}}    \bigg[ {F^{\epsilon}_{ (2-n)p_+p_--\nu sp_+,\nu rp_-}(\tau)}^\prime
    +{F^{\epsilon}_{(2-n)p_+p_--\nu rp_-,\nu sp_+}(\tau)}^\prime \\
  &\qquad+ \left.\left(n+2\frac{\nu s}{p_-}\right)F^{\epsilon}_{(2-n)p_+p_--\nu
      sp_+,\nu rp_-}(\tau)+
    \left(n+2\frac{\nu r}{p_+}\right)F^{\epsilon}_{ (2-n)p_+p_--\nu rp_-,\nu sp_+}(\tau)
\right].
\end{split}
\end{equation}
where \(1\leq r<p_+,1\leq s<p_-\).
Note that we do not regularise the characters of the irreducible Virasoro modules \(\VirIrr{r,s}\).

\section{Modular-type properties of regularised characters} 

In this section we will develop the modular transformation properties of the
characters of both standard and atypical modules.

The transformation properties of standard modules can be found in
\cite[Proposition 22]{CM} and other papers. 
\begin{prop}\label{sec:FockStransf}
The modular $S$-transformation of typical characters is
\begin{equation*}
  \begin{split}
    \mathrm{ch}[ \mathcal{F}^\epsilon_{\lambda+\alpha_0/2}]\left(-\frac{1}{\tau}\right)&=
    \int_{\mathbb R} S_{\lambda+\alpha_0/2}^\epsilon(x) \mathrm{ch}[\mathcal{F}^\epsilon_{x+\alpha_0/2}](\tau)dx,
  \end{split}
\end{equation*}
with $S_{\lambda+\alpha_0/2}^\epsilon(x)= e^{2\pi\epsilon (\lambda-x)}e^{-2\pi\lambda x} $.
\end{prop}
Next we consider the modular properties of characters of atypical modules. The
easiest to consider are the characters of the minimal model representations
\(\VirIrr{r,s}\) for which it is well known that they form a finite dimensional
representation of the modular group without the need to invoke any
regularisation or partial theta functions.
\begin{prop}
  \begin{displaymath}
    \ch{\mathcal{L}_{r,s}}\left(-\frac{1}{\tau}\right)=\sum_{(r', s') \in \mathcal{T}_{p_+, p_-}} \SmatMM{r,s}{r',s'} \ch{\mathcal{L}_{r',s'}}(\tau),
  \end{displaymath}
where
\begin{equation} \label{S-matrix}
  \SmatMM{r,s}{r',s'}= (-1)^{(r+s)(r'+s')} \sqrt{\frac{8}{p_+ p_-}} \sin\left(
    \frac{\pi r r' (p_- -p_+)}{p_+}\right) \sin\left( \frac{\pi s s' (p_--p_+)}{p_-}\right),
\end{equation}
and
\begin{displaymath}
  \mathcal{T}_{p_+, p_-}=\left\{(r,s)\left|\ 1 \leq r \leq p_+ - 1,1 \leq s \leq p_- - 1, s p_+ > r p_-\right.\right\},
\end{displaymath}
Asymptotically, as $y \rightarrow 0^+$,
\begin{displaymath}
\ch{\mathcal{L}_{r,s}}(iy) \sim \SmatMM{r,s}{r_0,s_0} e^{ \frac{\pi}{12 y}(1-\frac{6}{p_+ p_-})},
\end{displaymath}
where $(r_0,s_0)\in\KacT{p_+,p_-}$ is the unique pair such that $r_0  p_- - s_0 p_+=1$,
that is, \((r_0,s_0)\) is the label of the Virasoro highest weight module with
least conformal dimension.
\end{prop}
\begin{proof}
The first formula is the well known \(S\)-transformation of the characters of
the Virasoro minimal models which is easily derived from the transformation
properties of theta functions.
For the second formula it is sufficient to observe that the character with the $(r_0,s_0)$ label dominates 
in the expansion of $\ch{\mathcal{L}_{r,s}}(-\frac{1}{\tau})$.
\end{proof}

In order to compute the modular transformations of the regularised characters
of the remaining atypical irreducible modules, we need to understand those of $F^\epsilon_{b, c}$.
\begin{prop}\label{prop:S-trafo}  Let $\epsilon \notin i \mathbb{R}$.
The $S$-transformation of the mixed false theta functions \(F_{b,c}^\epsilon\) is given by 
\begin{equation}\nonumber
\begin{split}
F_{ b, c}^\epsilon\left(-\frac{1}{\tau}\right) &=  \int_{\mathbb
  R}\SmatFT{\epsilon}{b,c}{x} \ch{\FF{x+\frac{\alpha_0}{2}}^\epsilon}(\tau) \dd x +
\frac{1-\sgn\re(\epsilon)}{2}  X^\epsilon_{b, c}(\tau)
\end{split}
\end{equation}
with the ``correction term''
\begin{displaymath}
X_{b, c}^\epsilon = \frac{iq^{-\frac{\epsilon^2}{2}}}{\sqrt{2p_+p_-}}
\sum_{m=0}^{2p_+p_--1} e^{-\pi i\frac{bm}{p_+p_-}} \sin\left(\pi \frac{cm}{p_+p_-}\right) \frac{\theta_{p_+p_-,m}(i\sqrt{2p_+p_-}\epsilon\tau ,\tau)}{\eta(\tau)}
\end{displaymath}
and the $S$-kernel
\begin{displaymath}
\SmatFT{\epsilon}{b,c}{x} = - e^{-2\pi \epsilon x} e^{2\pi i
  \frac{\left(b-p_+p_-\right)}{\sqrt{2p_+p_-}}
  \left(x+i\epsilon\right)}\frac{\sin\left(2\pi c\frac{x+i\epsilon}{\sqrt{2p_+p_-}}\right)}
{\sin\left(\pi\sqrt{2p_+p_-} \left(x+i\epsilon\right)\right)}.
\end{displaymath}
\end{prop}
\begin{proof}
In \cite{CM} it was shown that
\begin{displaymath}
  P_\epsilon(u,\tau)=\sum_{k\geq0}e^{2\pi\epsilon(k+\frac12)}z^{k+\frac12}q^{\left(k+\frac12  \right)^2/2}
\end{displaymath}
transforms as
\begin{multline*}
  P_\epsilon\left(\frac{u}{\tau},\frac{-1}{\tau}\right)
  =\frac{e^{\pi i u^2/\tau}\sqrt{-i\tau}}{2}\left(-i
    \int_{\mathbb{R}}\frac{q^{x^2/2}z^x}{\sin(\pi(x+i\epsilon))}\dd x\right.\\
  \left.
    +\frac12(1+\sgn\re(\epsilon))\sum_{n\in\mathbb{Z}}(-1)^nz^{n-i\epsilon}q^{(n-i\epsilon)^2/2}
    \right).
\end{multline*}
Since
\begin{displaymath}
  P^\epsilon_{a,b}(u,\tau)=z^{\frac{b}{2a}-\frac12}e^{2\pi\epsilon\left(\frac{b}{2a}-\frac12\right)}
  q^{a\left(\frac{b}{2a}-\frac12\right)^2}P_\epsilon(u+(b-a)\tau,2a\tau),
\end{displaymath}
the transformation formulae for partial theta functions can obtained from that
of \(P_\epsilon(u,\tau)\).
\begin{displaymath}
  P^\epsilon_{a, b}\left(\frac{u}{\tau},\frac{-1}{\tau}\right)=
  e^{2\pi i\left(\frac{b}{2a}-\frac12\right)\frac{u}{\tau}}e^{2\pi\epsilon\left(\frac{b}{2a}-\frac12\right)}
  e^{-2\pi i\frac{a}{\tau}\left(\frac{b}{2a}-\frac12\right)^2}
  P_\epsilon\left(\frac{\tilde u}{\tilde \tau},\frac{-1}{\tilde\tau}\right),\quad
  \tilde u=\frac{u-b+a}{2a},\tilde\tau=\frac{\tau}{2a}.
\end{displaymath}
Algebraic manipulations thus yield
\begin{multline*}
  P_\epsilon\left(\frac{u}{\tau},\frac{-1}{\tau}\right)=
  e^{2\pi\epsilon\left(\frac{b}{2a}-\frac12\right)}e^{\pi
    i\frac{u^2}{2a\tau}}\sqrt{\frac{-i\tau}{8a}}
  \left(-i\int_{\mathbb{R}}\frac{q^{\frac{x^2}{4a}}z^{\frac{x}{2a}}e^{2\pi
        i\frac{a-b}{2a}x}}{\sin(\pi(x+i\epsilon))}\dd x\right.\\
  \left.
    +\frac{1+\sgn\re(\epsilon)}{2}\sum_{n\in\mathbb{Z}}(-1)z^{\frac{n-i\epsilon}{2a}}e^{2\pi
      i\frac{a-b}{2a}(n-i\epsilon)}
    q^{\frac{(n-i\epsilon)^2}{4a}}\right).
\end{multline*}
By plugging this transformation formula into 
definition of the mixed false theta
functions \(F_{b,c}^\epsilon\), rescaling the integration variable \(x\) by a
factor of \(\sqrt{2p_+p_-}\) and
simplifying, one sees that
\begin{multline*}
F_{b, c}^\epsilon\left(-\frac{1}{\tau}\right) =  \int_{\mathbb R}
\frac{q^{\frac{x^2}{2}}}{\eta(\tau)} e^{-2\pi i \frac{\left(b-p_+p_-\right)}{\sqrt{2p_+p_-}}
  \left(x-i\epsilon\right)}\frac{\sin\left(2c/\sqrt{2p_+p_-}\pi
    \left(x-i\epsilon\right)\right) }{\sin\left(\sqrt{2p_+p_-}\pi 
\left(x-i\epsilon\right)\right)} dx +\\
 +\frac{1-\sgn\re(\epsilon)}{\eta(\tau)\sqrt{32p_+p_-}}
 \sum_{k\in\mathbb Z} 
\left(e^{\pi i \frac{ck}{p_+p_-}}-e^{-\pi i \frac{ck}{p_+p_-}}\right)e^{-\pi i
  \frac{b}{p_+p_-}k}
q^{\frac{\left(k+i\sqrt{2p_+p_-}\epsilon\right)^2}{4p_+p_-}}.
\end{multline*}
The second summand can be rewritten in terms of standard theta functions:
\begin{displaymath}
\begin{split}
& \sum_{k\in\mathbb Z} \left(e^{\pi i \frac{ck}{p_+p_-}}-e^{-\pi i
    \frac{ck}{p_+p_-}}\right)e^{-\pi i \frac{bk}{p_+p_-}}
q^{\frac{\left(k+i\sqrt{2p_+p_-}\epsilon\right)^2}{4p_+p_-}} \\
&\quad=q^{-\frac{\epsilon^2}{2}} \sum_{m=0}^{2p_+p_--1} \left( e^{\pi i
    \frac{cm}{p_+p_-}}- e^{-\pi i \frac{cm}{p_+p_-}}\right)e^{-\pi i \frac{bm}{p_+p_-}}
\sum_{k\in\mathbb Z} q^{p_+p_-\left(k+\frac{m}{2p_+p_-}\right)^2} q^{i\sqrt{2p_+p_-}\epsilon\left(k+\frac{m}{2p_+p_-}\right)} \\
&\quad= q^{-\frac{\epsilon^2}{2}} \sum_{m=0}^{2p_+p_--1} \left( e^{\pi i
    \frac{cm}{p_+p_-}}- e^{-\pi i \frac{cm}{p_+p_-}}\right)e^{-\pi i \frac{bm}{p_+p_-}}
\theta_{p_+p_-, m}\left(i\sqrt{2p_+p_-}\epsilon\tau, \tau\right),
\end{split}
\end{displaymath}
where the second line follows by substituting \(k\in\mathbb{Z}\) by
\(2p_+p_-k+m,\ k\in\mathbb{Z},m=0,\dots,2p_+p_--1\).
The proposition then follows by changing the integration variable from $x$ to $-x$.
\end{proof}

We can now use Proposition \ref{prop:S-trafo} to determine modular
\(S\)-transformations of the remaining atypical submodules of standard modules.
\begin{thm} \label{S-atypical}
  For \(1\leq r\leq p_+,1\leq s\leq p_-\), the modular $S$-transformation of the characters of the 
\(\SingIm{r,s;n}\) modules is
  \begin{equation}\nonumber
    \begin{split}
      \ch{\SingIm{r,s;n}^{\epsilon}}\left(-\frac{1}{\tau}\right) 
      &=n\delta_{n\geq 0} \sum_{(r', s') \in \KacT{p_+,
          p_-}}(-1)^{n(p_+s'+p_-r')} S^{{\rm
          Vir}}_{(r,s),(r',s')}\ch{\mathcal{L}_{r',s'}}(\tau)\\
      &\quad + 
      \int_{\mathbb R}S_{r,s;n}^{\epsilon}(x)
      \ch{\FF{x+\frac{\alpha_0}{2}}^\epsilon}(\tau) dx +
      \frac{1-\sgn\left(\re(\epsilon)\right)}{2} Y_{r,s;n}^\epsilon(\tau)\\
    \end{split}
  \end{equation}
  with $S$-kernel
  \begin{equation}\nonumber
    \begin{split}
      S_{r,s;n}^{\epsilon}(x) &=-e^{-2\pi \epsilon x} e^{-\pi i n\sqrt{2p_+p_-}(x+i\epsilon)}
      \frac{\sin\left(2\pi rp_-\frac{x+i\epsilon}{\sqrt{2p_+p_-}}\right)
        \sin\left(2\pi sp_+\frac{x+i\epsilon}{\sqrt{2p_+p_-}}\right)}
    {\sin\left(\pi\sqrt{2p_+p_-}(x+i\epsilon) \right)^2},
    \end{split}
  \end{equation}
  and the correction term 
  \begin{equation}\nonumber
    \begin{split}
      Y_{r,s;n}^{\epsilon}(\tau) &=       \frac{1}{4\pi p_+p_-}\frac{d}{d
        \epsilon}q^{-\frac{\epsilon^2}{2}}\sum_{m=0}^{2p_+p_--1}
      (-1)^{mn}\sin\left(\pi\frac{rm}{p_+}\right)\sin\left(\pi\frac{sm}{p_-}\right)
      \frac{\theta_{p_+p_-,m}(i\sqrt{2p_+p_-}\epsilon\tau,\tau)}{\eta(\tau)}\\
      &\quad
      -\frac{nq^{-\frac{\epsilon^2}{2}}}{\sqrt{2 p_+p_-}}\sum_{m=0}^{2p_+p_--1}
      (-1)^{mn}\sin\left(\pi\frac{rm}{p_+}\right)\sin\left(\pi\frac{sm}{p_-}\right)
      \frac{\theta_{p_+p_-,m}(i\sqrt{2p_+p_-}\epsilon\tau,\tau)}{\eta(\tau)}\\
      &\quad
      +\frac{ri}{p_+}\frac{q^{-\frac{\epsilon^2}{2}}}{\sqrt{2 p_+p_-}}\sum_{m=0}^{2p_+p_--1}
      (-1)^{mn}\cos\left(\pi\frac{rm}{p_+}\right)\sin\left(\pi\frac{sm}{p_-}\right)
      \frac{\theta_{p_+p_-,m}(i\sqrt{2p_+p_-}\epsilon\tau,\tau)}{\eta(\tau)}\\
      &\quad
      +\frac{si}{p_-}\frac{q^{-\frac{\epsilon^2}{2}}}{\sqrt{2 p_+p_-}}\sum_{m=0}^{2p_+p_--1}
      (-1)^{mn}\sin\left(\pi\frac{rm}{p_+}\right)\cos\left(\pi\frac{sm}{p_-}\right)
      \frac{\theta_{p_+p_-,m}(i\sqrt{2p_+p_-}\epsilon\tau,\tau)}{\eta(\tau)}
    \end{split}
  \end{equation}
\end{thm}
\begin{proof}
The theorem follows  directly from applying Proposition \ref{prop:S-trafo} -- which gives
the modular transformations of the mixed false theta functions \(F^\epsilon_{b,c}\) -- to the
appropriate character formulae for \(\ch{\SingIm{r,s;n}^\epsilon}\),
\(\ch{\SingIm{r,p_-;n}^\epsilon}=\ch{\SingIm{r,p_-;n}^{\epsilon,+}}\),
\(\ch{\SingIm{p_+,s;n}^\epsilon}=\ch{\SingIm{p_+,s;n}^{\epsilon,-}}\)
and \(\ch{\SingIm{p_+,p_-;n}^\epsilon}=\ch{\FF{p_+,p_-;n}^\epsilon}\);
and by using the identities
\begin{displaymath}
  \begin{split}
    {F_{b,c}^\epsilon}^\prime &=
    -\frac{1}{\sqrt{2p_+p_-}}\frac{1}{2\pi}\frac{d}{d\epsilon}F_{b,c}^\epsilon,\\
    \SmatMM{r,p_--s}{r',s'}&=-(-1)^{p_+s'+p_-r'}\SmatMM{r,s}{r',s'}.
  \end{split}
\end{displaymath}
\end{proof}

\begin{remark}\label{cor:corVirS}
An interesting observation is that in the limit $\epsilon\rightarrow 0$, the correction term \(Y_{r,s;n}^\epsilon(\tau)\) of the characters of
the \(\SingIm{r,s;n}\) tends to
\begin{equation}\nonumber
\lim_{\epsilon\rightarrow 0} Y_{r,s;n}^{\epsilon}(\tau) = -\frac{n}{2}
\sum_{(r', s')\in \KacT{p_+, p_-}} (-1)^{n(p_+s'+p_-r')} S^{\text{Vir}}_{(r,s),(r',s')} \ch{\VirIrr{r',s'}}(\tau).
\end{equation}
We thus observe that the correction term carries information about the minimal
model characters.
\end{remark}

\section{Quantum Dimensions}

As explained in the introduction and as in \cite{CM}, we define the
regularised quantum dimension of a module \(M\) to be
\begin{displaymath}
\qdim{M^\epsilon} := \lim_{y \rightarrow 0+}\frac{\ch{M^{\epsilon}}(i y )}{\ch{V^\epsilon}(i y )},
\end{displaymath}
where  $V$ is the vertex operator
algebra itself. Then, the ordinary (non-regularised) quantum dimensions 
can be computed as the left and right limit
\begin{displaymath}
\qdim{V}^\pm := \lim_{\epsilon \rightarrow 0 \pm} \qdim{V^\epsilon},
\end{displaymath}
where $0$ is approached from the left or the right along the real axis.  A priori it is not clear that
these two limits should agree.
In order to be able to compute quantum dimensions we will make use of the
following trick commonly used in rational conformal field theories
\cite{DV}:
\begin{displaymath}
  \qdim{M^\epsilon}=\lim_{y \rightarrow 0+}\frac{\ch{M^{\epsilon}}(i y )}{\ch{V^\epsilon}(i y )}
  = \lim_{y \rightarrow  +\infty}\frac{\ch{M^{\epsilon}}(-1/iy)}{\ch{V^\epsilon}(-1/iy)}.
\end{displaymath}
That is, we can use modular transformation formulae of regularised characters
and the quantum dimension is the ratio of dominating terms in the numerator
and denominator.
In rational theories, this method results in $\qdim{M}=\frac{S_{i,{\rm min}}}{S_{0,{\rm min}}}$, where
${\rm min}$ denotes the label corresponding to the highest weight module of least conformal dimension. 
If all non-vacuum conformal dimensions are positive, 
then ${\rm min}$ is the label that denotes the vertex operator algebra itself (for this and more general constructions see \cite{BM}).
Since the modular transformations of characters considered in this paper
involve integrals over continua of modules, we make use of the 
following convenient fact used in asymptotic analysis \cite{special}.
\begin{lemma} \label{asymp-int} Let $f(x)$ satisfy the properties in \cite{special},  and
 let $F(y)=\int_{0}^\infty e^{-y \pi x^2} f(x) d x$.
Then asymptotically, as $y \rightarrow +\infty$, 
\begin{displaymath}
F(y)= \frac{1}{2\sqrt{y}} f(0)+O\left(\frac{1}{{y}}\right).
\end{displaymath}
\end{lemma}

\subsection{Quantum dimensions of $\SingVOA{p}$-modules}\label{sec:1,p}

Before we dive into a thorough discussion of the $\SingVOA{p_+,p_-}$ algebra, 
we first discuss the $\SingVOA{p}$ algebra. We specialise our notation to
facilitate comparisons to \cite{CM},
where the first two authors previously studied  regularised quantum
dimensions of $\SingVOA{p}$ modules for $\re(\epsilon)>0$. For the remainder
of this section let \(p_+=1,p_-=p\) and denote
the irreducible atypical modules by \(\SingImP{r,s}\cong \SingIm{1,s;r-1}, r\in \mathbb{Z},
1\leq s\leq p\).
Looking at the modular transformation of regularised characters, we observe that the correction term dominates provided that the real part of epsilon is smaller than
\begin{equation}\label{eq:regime}
B_\epsilon^p:= - \mathrm{min}\left\{ \Big|\frac{m}{\sqrt{2p}}-\mathrm{Im}\left(\epsilon\right)\Big|\  \Big\vert \ m\in\mathbb Z\setminus p\mathbb Z\ \right\}. 
\end{equation}
Otherwise the continuous part dominates. 

Using Lemma \ref{asymp-int} one can then easily compute (see also \cite{CM}):
\begin{prop} 
  For \(\re(\epsilon)>B_\epsilon^p\), $\re(\epsilon) \neq 0$, the quantum dimensions of typical and atypical
  modules are:
  \begin{equation} \label{cm-qdim}
    \begin{split}
      \qdim{\FF{\lambda}^\epsilon}&=e^{2\pi\epsilon(\lambda-\alpha_0/2)}\frac{\sin\left(i \pi\sqrt{2p} \epsilon\right)}{\sin\left(i \pi\frac{ \epsilon}{\sqrt{2p}}\right)},\\
      \qdim{\SingImP{r,s}^\epsilon}&=e^{\pi \epsilon (r-1)\sqrt{2p}}
      \frac{\sin\left(2\pi \frac{s \epsilon
            i}{\sqrt{2p}}\right)}{\sin\left(2\pi \frac{\epsilon
            i}{\sqrt{2p}}\right)},
    \end{split}
  \end{equation}
  where $\alpha_0=\sqrt{2p}-\sqrt{2/p}$.
\end{prop}
\noindent
Observe that $$\qdim{M_{r,s}^{\epsilon+i \sqrt{2p}}}=\qdim{M_{r,s}^{\epsilon}},$$ 
so the quantum dimension is defined on a semi-infinite cylinder.

We now consider the case where the correction term dominates and $\epsilon$ is real. For $\tau \rightarrow
+ i\infty$, $X_{r,s}(\tau)$ is dominated by the terms with $k=0, m=1$ and
$k=-1, m=2p-1$.
\begin{equation} \label{1sp}
\begin{split}
  \qdim{\SingImP{r,s}^\epsilon} &= \lim_{\tau\rightarrow 0+}
  \frac{\ch{\SingImP{r,s}^{\epsilon}}(\tau)}{\ch{\SingImP{1,1}^\epsilon}(\tau)}
  = \lim_{\tau\rightarrow + i\infty}\frac{\ch{\SingImP{r,s}^{\epsilon}}(-1/\tau)}{\ch{M^\epsilon_{1,1}}(-1/\tau)} \\
  &= \lim_{\tau \rightarrow + i\infty}
  \frac{(-1)^r\sin\left(\pi\frac{s}{p}\right)
        \left(q^{\frac{i\epsilon}{\sqrt{2p}}}-q^{-\frac{i\epsilon}{\sqrt{2p}}}\right)q^{\frac{1}{4p}}}
      {(-1)\sin\left(\pi\frac{1}{p}\right)\left(q^{\frac{i\epsilon}{\sqrt{2p}}}-q^{-\frac{i\epsilon}{\sqrt{2p}}}\right)q^{\frac{1}{4p}}}
  =(-1)^{r-1}\frac{\sin(\pi s/p)}{\sin(\pi/p)}.
\end{split}
\end{equation}
Staying within the $\re(\epsilon)<B_\epsilon^p$ regime, consider the (open) strips
\begin{displaymath}
  \strip{k,m}=\left\{\epsilon\in\mathbb{C}\left| k+\frac{2m-1}{4p}<
      \frac{\im(\epsilon)}{\sqrt{2p}} 
      < k+\frac{2m+1}{4p} \right.\right\},
\end{displaymath} 
$k\in\mathbb{Z}$ where $m=0,\dots,2p-1$. If \(\epsilon\in \strip{k,m}\) then the dominating term of \(Y_{1,s;r-1}\) has
index \(k,m\), however if \(m=0\) or \(m=p\) then
\(\sin\left(\pi\frac{sm}{p}\right)=0\)
and the neighbouring indices dominate.

\begin{prop} \label{qdim-1p-disc}  
For \(\epsilon\in\strip{k,m}, k\in\mathbb{Z},m=0\dots,2p-1\),
\begin{equation} \label{main-qdim}
\qdim{\SingImP{r,s}^{\epsilon}}=
\begin{cases}
  (-1)^{m(r-1)} \displaystyle{\frac{\sin(\pi m s/p)}{\sin(\pi m /p)}}&\text{if } m\neq0,p,\\
  (-1)^{(m+1)(r-1)+\frac{m}{p}(s-1)} \displaystyle{\frac{\sin(\pi s/p)}{\sin(\pi
    /p)}}&\text{if } m=0,p.
\end{cases}
\end{equation}
\end{prop}
An important point to make here is that for 
\(n=2p k+m,k\in\mathbb{Z},m=0,\dots 2p-1\),
\begin{equation} \label{left-right}
\lim_{\epsilon \rightarrow \left(\frac{i n}{\sqrt{2p}}\right)^+} \qdim{\SingImP{r,s}^\epsilon}=(-1)^{(r-1)m} \frac{\sin(\pi n s/p)}{\sin(\pi n /p)}
=\lim_{\epsilon \rightarrow \left(\frac{i n}{\sqrt{2p}}\right)^-} \qdim{\SingImP{r,s}^\epsilon},
\end{equation}
if \(m\neq 0, p\), 
where the limits were taken along the line parallel to the real axis and the
\(\pm\) of \(\left(\frac{i n}{\sqrt{2p}}\right)^\pm\) indicates the sign of \(\re(\epsilon)\).
Thus, the quantum dimension ``leaks'' across the ``wall'' given by \eqref{eq:regime} from the continuous to discrete regime on a countable set.

For the  quantum dimensions of typicals modules in the regime
$\re(\epsilon)<B_\epsilon^p$ and \eqref{eq:regime}, the denominator dominates and thus
\begin{displaymath}
\qdim{\FF{\lambda}^\epsilon}=0.
\end{displaymath} 
Again, for \(n=2p k+m,k\in\mathbb{Z},m=0,\dots 2p-1\)
we can compare limits.
\begin{displaymath}
\lim_{\epsilon \rightarrow \left(\frac{n i}{\sqrt{2p}}\right)^+} \qdim{\FF{\lambda}^\epsilon}=0=\lim_{\epsilon \rightarrow \left(\frac{n i}{\sqrt{2p}}\right)^-} \qdim{\FF{\lambda}^\epsilon},
\end{displaymath}
if \(m\neq 0, p\). 

In \cite{CGP,BCGP}, an infinite dimensional (unrolled) quantum group at even root of unity $\bar{U}^H_{q}(sl_2)$, $q=e^{\pi i/p}$  was studied. 
The category of finite-dimensional weight modules for this quantum group is expected to be equivalent to a certain (tensor) subcategory of modules for the $\SingVOA{p}$-singlet.
One strong piece of evidence in support of this belief is the agreement of
fusion products among irreducibles \cite{CGP}.
Following the notation from \cite{BCGP}, irreducible modules are denoted by
$S_i \otimes \mathbb{C}_{pr}^H$, $i=0,...,p-1$, $r \in \mathbb{Z}$ (atypicals)
and $V_{\alpha}$, $\alpha \in \mathbb{C} \setminus \mathbb{Z} \cup r
\mathbb{Z}$ (typicals). The $M_{r,i+1}$ should then correspond to $S_i \otimes \mathbb{C}_{pr}^H$,
and $V_{\alpha}$ to $F_{\alpha}$.
Rather than pushing this connection further, we will instead compare our model
to the $A_1^{(1)}$ WZW models.

\begin{figure}
\begin{tikzpicture}[scale=0.8]

\shade [upper right=green, upper left=white!90!gray, lower right=green, lower left=white!80!gray] (-1.2,4)  rectangle (8,4.5);
\shade [upper right=green,  lower left=white!70!gray, lower right=green, upper left=white!80!gray] (-1.2,3) rectangle (8,4.03);
\shade [upper right=green, upper left=white!70!gray, lower right=green, lower left=white!50!gray] (-1.2,2) rectangle (8,3.03);
\shade [upper right=green, upper left=white!50!gray, lower right=green, lower left=white!30!gray] (-1.2,1) rectangle (8,2.03);
\shade [upper right=green, upper left=white!30!gray, lower right=green, lower left=white!10!gray] (-1.5,0) rectangle (8,1.03);
\shade [upper right=green, upper left=white!10!gray, lower right=green, lower left=white!30!gray] (-1.5,-1) rectangle (8,0.03);
\shade [upper right=green, upper left=white!30!gray , lower right=green, lower left=white!50!gray] (-1.2,-2) rectangle (8,-0.97);
\shade [upper right=green, upper left=white!50!gray , lower right=green, lower left=white!70!gray] (-1.2,-3) rectangle (8,-1.97);
\shade [upper right=green, upper left=white!70!gray , lower right=green, lower left=white!80!gray] (-1.2,-4) rectangle (8,-2.97);
\shade [upper right=green, upper left=white!80!gray , lower right=green, lower left=white!90!gray] (-1.2,-4.5) rectangle (8,-3.97);

\node at (-4.5,0.5) {$\re(\epsilon)<0$};

\node at (1,0)  {$\epsilon=0$};
\node at (1,1.5) {$\epsilon=\frac{3i}{2\sqrt{2p}}$};
\node at (1,2.5) {$\epsilon=\frac{5 i}{2\sqrt{2p}}$};
\node at (1,3.5) {$\epsilon=\frac{7 i}{2\sqrt{2p}}$};
\node at (1,-1.5) {$\epsilon=\frac{-3i}{2\sqrt{2p}}$};
\node at (1,-2.5) {$\epsilon=\frac{-5 i}{2\sqrt{2p}}$};
\node at (1,-3.5) {$\epsilon=\frac{-7 i}{2\sqrt{2p}}$};

\node at (4.5,0.5) {$\re(\epsilon)>0$};
\draw [thin,dashed] (0,-4.5) -- (0,4.5);
\fill [black] (0,4) circle (2 pt);
\fill [black] (0,3) circle (2 pt);
\fill [black] (0,2) circle (2 pt);
\fill [black] (0,1) circle (2 pt);
\fill [black] (0,-4) circle (2 pt);
\fill [black] (0,-3) circle (2 pt);
\fill [black] (0,-2) circle (2 pt);
\fill [black] (0,-1) circle (2 pt);

\draw [thin, fill=white!80!gray] (-8.0,4.5)--(-0.5,4.5)--(0,4) --(-0.5,3.5)--(-8,3.5);
\draw [thin,fill=white!70!gray] (-8.0,3.5)--(-0.5,3.5)--(0,3)--(-0.5,2.5)--(-8,2.5);
\draw [thin,fill=white!50!gray] (-8.0,2.5)--(-0.5,2.5)--(0,2)--(-0.5,1.5) -- (-8,1.5);
\draw [thin,fill=white!30!gray] (-8.0,1.5)--(-.5,1.5)--(0,1)--(-1,0) -- (-8,0);
\draw [thin,fill=white!30!gray] (-8.0,0)--(-1,0) -- (0,-1) -- ( -0.5,-1.5) -- (-8,-1.5);
\draw [thin,fill=white!50!gray] (-8.0,-1.5)--(-0.5,-1.5)--(0,-2)--(-0.5,-2.5) -- (-8,-2.5);
\draw [thin,fill=white!70!gray] (-8.0,-2.5)--(-0.5,-2.5)--(0,-3)--(-0.5,-3.5)--(-8,-3.5);
\draw [fill=white!80!gray] (-8.0,-3.5)--(-0.5,-3.5)--(0,-4) -- (-0.5,-4.5) -- (-8,-4.5);

\end{tikzpicture}
\caption{Quantum dimensions of atypical $\SingVOA{p}$ modules in the $\epsilon$-plane
  \medskip\newline
{\small
  The dashed line is the imaginary axis; the right hand-side of the diagram represents the continuous region, where the 
  quantum dimension is given by formula (\ref{cm-qdim}); and the
  semi-infinite strips on the left represent regions where the quantum dimensions are constant. The black dots 
   at $\epsilon=\frac{n i}{\sqrt{2p}}$, $n \in \mathbb{Z}$, denote
  ``drip'' points where the left and  right limits of quantum dimensions coincide.}}
\label{fig:wall}
\end{figure}

\begin{lemma} \label{su2} Let $P_+^k=\{0,1...,k\},\ k\in\mathbb{N}$ label the irreducible
  modules of the $A_1^{(1)}$ vertex operator algebra at level $k \in \mathbb{N}$. Then the associated 
$S$-matrix is given by 
\begin{displaymath}
  S_{a,b}^{k}=\sqrt{\frac{2}{k+2}}  \sin \left( \pi
    \frac{(a+1)(b+1)}{k+2}\right),\quad a,b\in P_+^k,
\end{displaymath}
where the zeroth label corresponds to the vacuum module vertex operator algebra $L_{sl_2}(k, \Lambda_0)$.
\end{lemma}

The results of this section then imply the following proposition.
\begin{prop} \label{fusion-1p}
There is an induced ring structure on the space of quantum dimensions of
$\SingVOA{p}$ modules. 
For $\re(\epsilon) <B_\epsilon^p$, this ring is isomorphic to the fusion ring of $A_1^{(1)}$ at level $p-2$.
\end{prop}
\begin{proof}
Formula (\ref{left-right}) of Proposition \ref{qdim-1p-disc} implies that the
map $[X]\mapsto\qdim{X^\epsilon}$, where \(X\) is any \(\SingVOA{p}\) singlet
module,
is a ring homomorphism for all values of \(\epsilon\). We now specialise to 
\(\re(\epsilon)<B_\epsilon^p\) where, by Proposition \ref{qdim-1p-disc}, the space of quantum dimensions is
``strip wise'' constant.
Next we analyse the image of the ring homomorphism. Observe that only
(\ref{main-qdim}) is relevant here and also the periodicity 
$n \ \equiv i \ {\rm mod} \  2p$, so we need only consider the values in (\ref{main-qdim}) where $n=1,...,p-1,p+1,...,2p-1$. We also distinguish between $r$ even and $r$ odd. 

By Proposition 18 for $r$ odd and $n=1,...,p-1,p+1,...,2p-1$ we get

\begin{displaymath}
A({\rm odd},s)=\left\{ \frac{\sin(\frac{\pi s }{p})}{\sin(\frac{\pi}{p})},..., \frac{\sin(\frac{\pi (p-1) s }{p})}{\sin(\frac{(p-1)\pi}{p})}, (-1)^s \frac{\sin(\frac{\pi s }{p})}{\sin(\frac{\pi(p+1)}{p})},...,(-1)^s \frac{\sin(\frac{\pi (p-1) s }{p})}{\sin(\frac{(2p-1)\pi}{p})}. \right\}
\end{displaymath}

For $r$ even, and $n=1,...,p-1,p+1,...,2p-1$, we have 

\begin{displaymath}
A({\rm even},s)=\biggl\{-\frac{\sin(\frac{\pi s }{p})}{\sin(\frac{\pi}{p})},.., (-1)^{p-1} \frac{\sin(\frac{\pi (p-1) s }{p})}{\sin(\frac{(p-1)\pi}{p})}, 
\end{displaymath}
\begin{displaymath}
 (-1)^{p+1+s} \frac{\sin(\frac{\pi s }{p})}{\sin(\frac{\pi(p+1)}{p})},...(-1)^{p+s+i}...,(-1)^{p+p-1}(-1)^s \frac{\sin(\frac{\pi (p-1) s }{p})}{\sin(\frac{(2p-1)\pi}{p})}.\biggr\}
\end{displaymath}

These two sets are related by $A({\rm odd},p-s)=-A({\rm even},s)$. So we can
ignore the quantum dimensions coming 
from either $r$ even or $r$ odd. Finally we use the fact that the $(p-1)
\times (p-1)$ matrix $\frac{\sin( \frac{ \pi a b }{p})}{\sin(\frac{a \pi}{p})}$
is invertible.  So the image of the ring homomorphism is  $p-1$-dimensional.
\end{proof}

\subsection{Quantum dimensions of $\SingVOA{p_+,p_-}$-models}

For the remainder of this section we assume that \(p_\pm\geq2\).
Recall that the vacuum module character is given by 
\begin{equation*}
  \ch{\SingKer{1,1;0}^\epsilon}=\ch{\VirIrr{1,1}}+\ch{\SingIm{1,1;0}^\epsilon},
\end{equation*}
where, as mentioned previously, we do not regularise the Virasoro part since it
leads to divergent limits. 

In the limit
\begin{displaymath}
  \lim_{y\rightarrow +\infty} \ch{\SingKer{1,1;0}^\epsilon}\left(\frac{-1}{iy}\right)
\end{displaymath}
the minimal model terms are dominated by integration and
\(Y_{r,s;n}^\epsilon\) terms. Thus, in the denominator of
\(\SingVOA{p_+,p_-}\) quantum dimensions \(\ch{\SingKer{1,1;0}^\epsilon}\) can
be replaced by \(\ch{\SingIm{1,1;0}^\epsilon}\). This immediately implies that
the \(\SingVOA{p_+,p_-}\) quantum dimensions of the simple minimal module
modules vanish for all \(\epsilon\).
\begin{cor}
  The \(\SingVOA{p_+,p_-}\) quantum dimensions of minimal model modules vanish for all \(\epsilon\),
  that is,
  \begin{displaymath}
    \qdim{\mathcal{L}_{r,s}}=0.
  \end{displaymath}
\end{cor}

For the remaining quantum dimensions we first consider $\re(\epsilon)>B_\epsilon^{p_+p_-}$, that is,
when the correction term \(Y_{r,s;n}^\epsilon\) does not appear in the modular
transformation formula of the \(\SingIm{r,s;n}\).
\begin{prop} 
  For \(\re(\epsilon)>B_\epsilon^{p_+p_-}\), $\re(\epsilon) \neq 0$, the typical and non-Virasoro irreducible atypical \(\SingVOA{p_+,p_-}\) quantum
  dimensions are given by
  \begin{equation}\label{qdim-cont}
    \begin{split}
    \qdim{\FF{\lambda+\alpha_0/2}^\epsilon}&=
    \lim_{y\rightarrow+\infty}\frac{\ch{\FF{\lambda+\alpha_0/2}^\epsilon}}{\ch{\SingIm{1,1;0}^\epsilon}}=
    e^{2 \pi \lambda \epsilon}  \frac{\sin(i \pi \sqrt{2 p_+ p_-} \epsilon) \sin(i \pi \sqrt{2 p_+ p_-} \epsilon)}
    {\sin(i \pi \sqrt{2 p_+/p_-} \epsilon) \sin(i  \pi \sqrt{2 p_-/p_+} \epsilon)}.\\
      \qdim{\mathcal{I}^\epsilon_{r,s;n}}&=\lim_{y \rightarrow \infty+}
      \frac{\ch{\mathcal{I}^\epsilon_{r,s;n}}(-1/iy)}{\ch{\mathcal{K}^\epsilon_{1,1;0}}(-1/iy)}
      =\frac{S_{r,s;n}^{\epsilon}(0)}{S_{1,1;0}^{\epsilon}(0)}\\
      &=e^{-\sqrt{2 p_+ p_-} \pi n \epsilon} \frac{\sin(i r \pi \sqrt{2
          p_-/p_+} \epsilon)\sin(i s \pi \sqrt{2 p_+/p_-} \epsilon)}{\sin(i
        \pi \sqrt{2 p_-/p_+} \epsilon)\sin(i \pi \sqrt{2 p_+/p_-} \epsilon)}.
    \end{split}
  \end{equation}
\end{prop}
\begin{proof}
  The quantum dimensions follow directly from applying Proposition \ref{sec:FockStransf},
  Theorem \ref{S-atypical} and Lemma \ref{asymp-int}.
\end{proof}
\noindent
Clearly, in the limit $\epsilon \rightarrow 0$, we get
\begin{displaymath}
  \begin{split}
    \qdim{\mathcal{F}_{\lambda+\alpha_0/2}}&=p_+p_-,\\
    \qdim{\mathcal{I}_{r,s;n}}&=rs.
  \end{split}
\end{displaymath}

Next we consider \(\re(\epsilon)<B_\epsilon^{p_+p_-}\). 
\begin{cor}
  For \(\re(\epsilon)<B_\epsilon^{p_+p_-}\) the \(\SingVOA{p_+,p_-}\) quantum dimensions of
  standard modules vanish, that is,
  \begin{displaymath}
    \qdim{\FF{\lambda}^\epsilon}=0,\quad \forall\lambda\in\mathbb{C}.
  \end{displaymath}
\end{cor}
\begin{proof}
Since we are considering \(\re(\epsilon)<B_\epsilon^{p_+p_-}\), the numerator of $\qdim{\FF{\lambda}^\epsilon}$ 
is dominated by the \(Y_{1,1;0}^\epsilon\) term in the denominator.
\end{proof}
\noindent
We define the strip
\begin{displaymath}
  \strip{k,m}=\left\{\epsilon\in\mathbb{C}\left| k+\frac{2m-1}{4p_+p_-}<
      \frac{\im(\epsilon)}{\sqrt{2p_+p_-}} 
      < k+\frac{2m+1}{4p_+p_-} \right.\right\},
\end{displaymath}
where $k\in\mathbb{Z}$ and $m\in\{0, 1, \dots, 2p_+p_--1\}$ .
\begin{prop}\label{sec:atypqdims}
  For \(\re(\epsilon)<B_\epsilon^{p_+p_-}\)  and \(\epsilon\in \strip{k,m}\), \(k\in\mathbb{Z}\)
  and \(m\in\{0, 1, \dots, 2p_+p_--1\}\), then
  \begin{equation}\label{atypicals-qdim}
    \qdim{\SingIm{r, s; n}^\epsilon}=
    \begin{cases}
      \left(-1\right)^{mn} \dfrac{\sin\left(\frac{\pi
            m}{p_+}r\right)\sin\left(\frac{\pi
            m}{p_-}s\right)}{\sin\left(\frac{\pi
            m}{p_+}\right)\sin\left(\frac{\pi m}{p_-}\right)}& p_+,p_- \text{
        do not divide }m,\\
      s\left(-1\right)^{t(p_-n+1+s)} \dfrac{\sin\left(\frac{\pi
            tp_-}{p_+}r\right)}{\sin\left(\frac{\pi t p_-}{p_+}\right)}
      &\text{only } p_- \text{ divides } m=tp_-,\\
      r\left(-1\right)^{t(p_+n+1+r)} \dfrac{\sin\left(\frac{\pi
            tp_+}{p_-}s\right)}{\sin\left(\frac{\pi t p_+}{p_-}\right)}&
      \text{only } p_+ \text{divides } m=tp_+,\\
      \left(-1\right)^{n+m} \dfrac{\sin\left(\frac{\pi r}{p_+}\right)\sin\left(\frac{\pi s}{p_-}\right)}{\sin\left(\frac{\pi }{p_+}\right)\sin\left(\frac{\pi }{p_-}\right)}&m=0,p_+p_-,
    \end{cases}
  \end{equation}
  where \(1\leq r\leq p_+\), \(1\leq s\leq p_-\) and \(n\in\mathbb{Z}\).
\end{prop}
\begin{proof}
  The quantum dimensions are easily computed case by case using the
  transformation formulae of Theorem \ref{S-atypical}.

  1. case:
\begin{displaymath} 
  \begin{split}
    &\qdim{\SingIm{r, s; n}^\epsilon} = 
    \lim_{\tau\rightarrow i\infty} \frac{Y^\epsilon_{r, s; n}(\tau)}{Y^\epsilon_{1, 1; 0}(\tau)} \\
    &\quad= \lim_{\tau\rightarrow i\infty} 
    \frac{\frac{\dd}{\dd
        \epsilon}q^{-\frac{\epsilon^2}{2}}\sum_{m=0}^{2p_+p_--1}
      (-1)^{mn}\sin\left(\pi\frac{rm}{p_+}\right)\sin\left(\pi\frac{sm}{p_-}\right)
      \theta_{p_+p_-,m}(i\sqrt{2p_+p_-}\epsilon\tau,\tau)}{
      \frac{\dd}{\dd
        \epsilon}q^{-\frac{\epsilon^2}{2}}\sum_{m=0}^{2p_+p_--1}
      \sin\left(\pi\frac{m}{p_+}\right)\sin\left(\pi\frac{m}{p_-}\right)
      \theta_{p_+p_-,m}(i\sqrt{2p_+p_-}\epsilon\tau,\tau)}\\
    &\quad= \left(-1\right)^{mn} \frac{\sin\left(\frac{\pi m}{p_+}r\right)\sin\left(\frac{\pi m}{p_-}s\right)}{\sin\left(\frac{\pi m}{p_+}\right)\sin\left(\frac{\pi m}{p_-}\right)},
  \end{split}
\end{displaymath}

2. case:
\begin{displaymath} 
\begin{split}
  \qdim{\SingIm{r, s; n}^\epsilon} &= 
  \lim_{\tau\rightarrow i\infty} \frac{Y^\epsilon_{r, s; n}(\tau)}{Y^\epsilon_{1, 1; 0}(\tau)} \\
  &= \lim_{\tau\rightarrow i\infty}  
  \frac{\frac{si}{p_-}\sum_{m=0}^{2p_+p_--1}
    (-1)^{mn}\sin\left(\pi\frac{rm}{p_+}\right)\cos\left(\pi\frac{sm}{p_-}\right)
    \theta_{p_+p_-,m}(i\sqrt{2p_+p_-}\epsilon\tau,\tau)}{
    \frac{i}{p_-}\sum_{m=0}^{2p_+p_--1}
    \sin\left(\pi\frac{m}{p_+}\right)\cos\left(\pi\frac{m}{p_-}\right)
    \theta_{p_+p_-,m}(i\sqrt{2p_+p_-}\epsilon\tau,\tau)}\\
  &= s\left(-1\right)^{t(p_-n+1+s)} \frac{\sin\left(\frac{\pi tp_-}{p_+}r\right)}{\sin\left(\frac{\pi t p_-}{p_+}\right)}.
\end{split}
\end{displaymath}

3. case (analogous to the 2. case):

4. case: For \(m=0,p_+p_-\) the corresponding coefficients in \(Y_{r,s,;n}\)
vanish and so the \(m\pm1\) terms dominate.
\end{proof}
\noindent
When comparing \eqref{qdim-cont} and (\ref{atypicals-qdim}), we see that the two
limits \(\epsilon\rightarrow0\) agree on a discrete set of points for
\(r<p_+\) and \(s<p_-\):
\begin{equation}\label{leak-points}
\lim_{\epsilon \rightarrow \frac{i m}{\sqrt{2p_+ p_-}}^+} \qdim{\SingIm{r, s; n}^\epsilon}= \lim_{\epsilon \rightarrow \frac{i m}{\sqrt{2p_+ p_-}}^-} \qdim{\SingIm{r, s; n}^\epsilon}.
\end{equation}

\begin{thm} \label{fusion.ring}
The regularised quantum dimensions satisfy 
\begin{equation}\nonumber
\begin{split}
  \qdim{\FF{\lambda}^\epsilon}\qdim{\FF{\mu}^\epsilon} &= \sum_{j_+=0}^{p_+-1}\sum_{j_-=0}^{p_--1} \qdim{\FF{\lambda+\mu+j_+\alpha_++j_-\alpha_-}^\epsilon}\\
  \qdim{\SingIm{r, s; n}^\epsilon}\qdim{\FF{\mu}^\epsilon} &= \sum_{j_+=0}^{r-1}\sum_{j_-=0}^{s-1} \qdim{\FF{\lambda+\alpha_{r-2j_+, s-2j_-;n}}^\epsilon}
\end{split}
\end{equation}
\begin{equation}\nonumber
\begin{split}
\qdim{\SingIm{1, 1; m}^\epsilon}\qdim{\SingIm{r, s; n}^\epsilon} &= \qdim{\SingIm{r, s; m+n}^\epsilon}
\end{split}
\end{equation}
\begin{equation}\nonumber
\begin{split}
\qdim{\SingIm{2, 1; 0}^\epsilon}\qdim{\SingIm{r, s; n}^\epsilon} &= \begin{cases} 
\qdim{\SingIm{2, s; n}^\epsilon} & \qquad \text{if} \ r=1 \\
\qdim{\SingIm{r-1, s; n}^\epsilon}+\qdim{\SingIm{r+1, s; n}^\epsilon} &\qquad \text{if} \ 1<r<p_+ \\
\qdim{\SingIm{1, s; n-1}^\epsilon}+2\qdim{\SingIm{p_+-1, s; n}^\epsilon}+&\qquad \text{if} \ r=p_+ \\
\qdim{\SingIm{1, s; n+1}^\epsilon} & 
\end{cases}\\
\qdim{\SingIm{1, 2; 0}^\epsilon}\qdim{\SingIm{r, s; n}^\epsilon} &= \begin{cases} 
\qdim{\SingIm{r, 2; n}^\epsilon} & \qquad \text{if} \ s=1 \\
\qdim{\SingIm{r, s-1; n}^\epsilon}+\qdim{\SingIm{r, s+1; n}^\epsilon} &\qquad \text{if} \ 1<s<p_- \\
\qdim{\SingIm{r, 1; n-1}^\epsilon}+2\qdim{\SingIm{r, p_--1; n}^\epsilon}+ &\qquad \text{if} \ s=p_- \\
\qdim{\SingIm{r, 1; n+1}^\epsilon} &
\end{cases}.
\end{split}
\end{equation}
\end{thm}
\begin{proof}
This result follows from expanding 
\begin{displaymath}
\frac{\sin(ipx)}{\sin(ix)} = \sum_{j=0}^{p-1} e^{(p-1-2j)x},
\end{displaymath}
and the computations follow very closely those of Chapter 4 of \cite{CM} so we omit them here.
\end{proof}
One can check directly that the above formulae induce a ring structure (we provide another proof below).

\begin{remark} \label{compare-RW}
 As mentioned in the introduction, we expect that an appropriate generalisation of the Verlinde formula holds for many irrational and even non $C_2$-cofinite vertex operator algebras. 
Considering characters as algebraic distributions instead of as functions on
the upper-half  plane, in Section 4 \cite{RW}, the third author and D. Ridout
found a  Verlinde-type ring of characters $(\mathcal{V}_{ch},+,\times)$, where
$\mathcal{V}_{ch}$ is the free abelian group generated by irreducible
characters. 
We will prove in Theorem \ref{injective} that this ring agrees with the one in Theorem \ref{fusion.ring}, in the sense that 
\begin{displaymath}
\ch{X} \rightarrow \qdim{X^\epsilon},
\end{displaymath} 
defines an isomorphism of rings. But there is even a third point of view here. By using the Verlinde product of  {\em regularised} characters defined in  \cite{CM}, as in the case of  $\SingVOA{p}$-modules, we can define a ring of regularised $\SingVOA{p_+,p_-}$-characters. Again, this ring is isomorphic to $(\mathcal{V}_{ch},+,\times)$. 
\end{remark}

The previous remark gives more than sufficient evidence for the correctness of the following conjecture.
\begin{conj} The relations in Theorem \ref{fusion.ring} or Section 4 \cite{RW}, obtained at the level of characters, 
remain valid inside the Grothendieck ring 
of a suitable quotient category of $\SingVOA{p_+,p_-}$-singlet modules.
\end{conj}

\section{Fusion rings and fusion varieties}

In this part, we study the behaviour of the Verlinde algebra of characters as we vary the parameter  $\epsilon$ throughout the 
complex plane. We also explain that the $\epsilon$ parameter, used in \cite{CM} merely as a computational tool, is actually 
a uniformisation parameter of an interesting algebraic variety  - a further
indication of the conceptual importance of the parameter.

Let us start first with a brief discussion of rational theories. It is known that the fusion ring of a rational vertex operator algebra (obtained either as the Grothendieck ring 
in the category or as the Verlinde formula of characters) is of finite
rank. After an extension of scalars, we easily infer that the 
corresponding fusion algebra is semi-simple. Therefore, the set of maximal ideals is precisely of the size of the set of equivalence classes of irreducible modules. 
This algebra can be also be viewed as the algebra of functions of a zero-dimensional variety - the {\em fusion variety}. There are various conjectures 
about the realisation of this fusion ring in terms of generators and relations, especially in the case of WZW models \cite{Ge,BR}.

As discussed in Section 1.4, from the categorical perspective the categories of $\SingVOA{p_+,p_-}$- and $\SingVOA{p}$-modules are too large for practical 
purposes. That is why we prefer the category constructed from atypical blocks only. One important advantage of this category is that it admits only countably many 
equivalence classes of irreducible modules. As such, after we pass to its Grothendieck ring,  it has direct relevance to algebraic geometry because
it defines the ring of functions of an $n$-dimensional complex algebraic variety with $n \geq 1$.

The first statement of the next result follows easily from \cite{CM}, while the second statement is taken from  Section 4, \cite{RW}.
\begin{prop} \label{chebyshev}
The Verlinde algebra of atypical blocks  has the following presentation 
 \begin{itemize}
 \item[(i)] 
   \begin{displaymath}
      \mathbb{V}(1,p)=\frac{\mathbb{C}[X,Z,Z^{-1}]}
      {\langle U_{p}(\frac{X}{2})-U_{p-2}(\frac{X}{2})-Z-Z^{-1}\rangle},
    \end{displaymath}
  \item[(ii)]
    \begin{displaymath}
      \mathbb{V}(p_+,p_-)=\frac{\mathbb{C}[X,Y,Z,Z^{-1}]}
      {\langle U_{p_+}(\frac{X}{2})-U_{p_+-2}(\frac{X}{2})-Z-Z^{-1},U_{p_-}(\frac{Y}{2})-U_{p_--2}(\frac{Y}{2})-Z-Z^{-1}\rangle},
    \end{displaymath}
    if one identifies
    \begin{displaymath}
      X\leftrightarrow\ch{\SingIm{2,1;0}},\qquad
        Y\leftrightarrow\ch{\SingIm{1,2;0}},\qquad Z^{\pm1}\leftrightarrow\ch{\SingIm{1,1;\pm 1}},
    \end{displaymath}
 \end{itemize}   
 where the \(U_i\) are Chebyshev polynomials of the second kind. 
\end{prop}
Clearly,  $Z^{-1}$ can be eliminated from each presentation by introducing $z^+$ and $z^-$ variables with the relation $z^+ z^- - 1$. 
Observe that the above presentation also holds over integers. This is the main reason why we used  Chebyshev polynomials of the second kind
instead of the first kind.  

Next, we analyse varieties $\mathcal{X}_{1,p}$ and $\mathcal{X}_{p_+,p_-}$, whose rings of functions are $\mathbb{V}(1,p)$ and $\mathbb{V}(p_+,p_-)$,
respectively.

\begin{thm} \label{fusion-variety}For every $p \geq 2$, the fusion variety $\mathcal{X}_{1,p}$ is an irreducible rational curve with $p-1$ ordinary singular points, 
while $\mathcal{X}_{p_+,p_-}$ is a genus zero curve
with $\frac{(p_- - 1)(p_+ -1)}{2}$  singular points.
Moreover, these singularities can be viewed as pinched cycles obtained from a Riemann surface of higher genus.
\end{thm}
\begin{proof} For $\mathcal{X}_{1,p}$, observe that our variety can be described as the set of points $(x,z)$ such that $z \neq 0$ and 
\begin{equation} \label{curve-1p}
2zT_p\left(\frac{x}{2}\right)-z^2-1=0,
\end{equation}
where $T_n(x)=\frac{1}{2}(U_n(x)-U_{n-2}(x))$ are Chebyshev polynomials of the first kind. From the Jacobian we see that possible singular 
points  are $z=\pm 1$ and $U_{p-1}'(x/2)=0$. We have $2(p-1)$ solutions to
this equation: $(x_k,z)=(2\cos(\frac{k \pi}{p}),\pm 1)$, $k=1,...,p-1$. By plugging these into 
(\ref{curve-1p})  and using the fact that $T_p(x_k)=\pm 1$, we see that either
$(x_k,1)$ or $(x_k,-1)$ are a singular points, all together $p-1$. These are
all ordinary double points. By using the formula $T_n(\cos(\theta))=\cos(n
\theta)$ we easily see that $x(t)=\frac{1}{t}+t$, $z(t)=\frac{1}{t^n}$ is a
rational parametrisation of the curve and it is thus of genus zero and clearly irreducible.

For $\mathcal{X}_{p_+,p_-}$ a similar analysis applies. Observe that this curve is defined as all $(x,y,z)$ such that $z \neq 0$ and $2zT_{p_+}(\frac{x}{2})-z^2-1=2zT_{p_-}(\frac{x}{2})-z^2-1=0$.
For the singular locus we first solve $z=\pm 1$, $U_{p_+-1}(\frac{x}{2})=U_{p_- -
  1}(\frac{y}{2})=0$, resulting in $2(p_+-1)(p_--1)$ possible singular points. As
these points are supposed to lie on the curve we have to eliminate everything
but $\frac{1}{2}(p_+-1)(p_--1)$ points.
The genus is again zero due to the rational 
parametrisation $x(t)=\frac{1}{t^{p_-}}+t^{p_-}$,  $y(t)=\frac{1}{t^{p_+}}+t^{p_+}$, $z(t)=\frac{1}{t^{p_+ p_-}}$.
\end{proof}
\begin{remark}
  Based on analysis of $\mathcal{X}_{p_+,p_-}$ for low $p_+$ and $p_-$, we
  believe that $\mathcal{X}_{p_+,p_-}$ is also irreducible.  In fact, if we
  extend the definition of $\mathcal{X}_{p_+,p_-}$ to all positive integers,
  the curve seems to be irreducible precisely when $p_+$ and $p_-$ are
  relatively prime.
\end{remark}

\begin{figure}
\begin{tikzpicture}[scale=0.7]
\filldraw[upper left=gray!80!white, upper right=green , lower left=gray!80!white, lower right=green] (0,0)  ellipse [x radius=6 cm, y radius=3 cm, start angle=0, end angle=180];
\filldraw[white] (-4,0) ellipse [x radius=2cm, y radius=1 cm];
\filldraw[white] (0,0) ellipse [x radius=2cm, y radius=1 cm];
\draw (0,1) arc (270:90:0.5cm and 1cm);
\draw (0,-1) arc (90:270:0.5cm and 1cm);
\draw[dotted] (0.5,2) arc (0:360:0.5cm and 1cm);
\draw[dotted] (0.5,-2) arc (0:360:0.5cm and 1cm);
 \fill [black] (-6,0) circle (2pt);
 \fill [black] (-2,0) circle (2pt);
\end{tikzpicture}
\caption{Fusion variety of the $(2,5)$-singlet algebra.}
\end{figure}
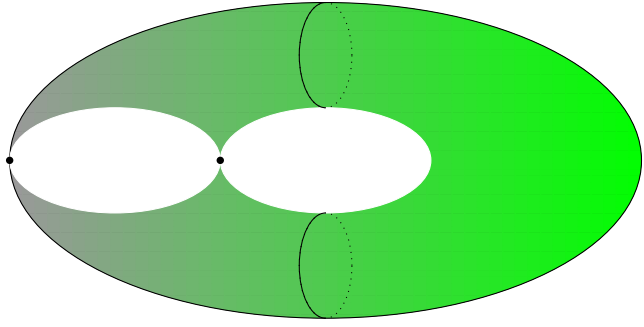

Now we see that quantum dimensions in $\re(\epsilon)>B_\epsilon^{p_+p_-}$ (at least those associated to generators of the fusion ring) provide a uniformisation of the fusion 
ring with $\epsilon$ being the uniformisation parameter. For example, by using formula (\ref{qdim-cont}), we easily see that
$\qdim{\mathcal{I}^\epsilon_{1,2;0}}$, $\qdim{\mathcal{I}^\epsilon_{2,1;0}}$, $\qdim{\mathcal{I}^\epsilon_{1,1; 1}}$
give a parametrisation mentioned the proof of Theorem \ref{fusion-variety} with
$t=e^{\frac{\pi \epsilon}{\sqrt{2p_+ p_-}}}$.
So $\epsilon$ should be viewed as a uniformisation parameter of the fusion variety.

\begin{thm} \label{injective}
The map 
\begin{align*}
  \ch{M} \mapsto \qdim{M^\epsilon}
\end{align*}
from the Verlinde ring of characters (as computed in \cite{RW}) to the space of quantum dimensions as functions of $\epsilon$ is:
\begin{enumerate} 
\item\label{item:1}  For the domain $\re(\epsilon) >B_\epsilon^{p_+p_-}$, $\re(\epsilon) \neq 0$: a ring isomorphism 
\item For the domain $\re(\epsilon) <B_\epsilon^{p_+p_-}$ 
a ring homomorphism with image isomorphic to the fusion ring of the $(p_+, p_-)$ Virasoro minimal model.
\item For the domain $\re(\epsilon) <B_\epsilon^{p_+p_-}$ and $m\neq 0$ divisible by $p_-$: 
a ring homomorphism with image isomorphic to  the fusion ring of $A_1^{(1)}$ at level $p_+-2$.
\item For the domain $\re(\epsilon) <B_\epsilon^{p_+p_-}$ and  $m\neq 0$ divisible by $p_+$: 
a ring homomorphism with image isomorphic to the fusion ring of $A_1^{(1)}$ at level $p_- -2$. 
\end{enumerate}
\end{thm}
\begin{proof}
(1) Comparing with \cite{RW} we see that the map is a ring homomorphism.
Also, it can be easily seen that the space of characters is spanned by
$\ch{\FF{\lambda}^\epsilon}, \ch{\SingIm{r, s; 0}^{\pm,\epsilon}}$
and $\ch{\SingIm{r, s; 0}^{\epsilon}}$, and that 
the space of characters in \cite{RW} is spanned by characters with the same labels.
The corresponding quantum dimensions are linearly independent since there are
no relations between the functions $e^{2\pi \epsilon \lambda}$, so we have an
isomorphism.

 (2) We first compare quantum dimensions in $\re(\epsilon)>B_\epsilon^{p_+p_-}$ to those in $\re(\epsilon)<B_\epsilon^{p_+p_-}$ and observe that the values of quantum dimension at each semi-infinite strip can be computed as the limiting value at a certain point in $\re(\epsilon)>B_\epsilon^{p_+p_-}$; see formula (\ref{leak-points}). Limits of course preserve products of quantum dimensions. Because of (\ref{item:1}), we naturally get a ring structure on the same span of quantum dimensions. 
Here, the kernel contains the ideal generated by the $\ch{\FF{\lambda}^\epsilon}, \ch{\SingIm{r, s; n}^{\pm,\epsilon}}$ and $\ch{\SingIm{r, s; 0}^{\epsilon}}-\ch{\SingIm{p_+-r,p_-- s; 0}^{\epsilon}}$. The product of the quantum dimensions of the $\SingIm{r, s; 0}^{\epsilon}$ for $r, s $ in $\KacT{p_+, p_-}$ is given by the corresponding Virasoro minimal model fusion coefficients. 
So that they must span a homomorphic image of the Virasoro minimal model fusion ring. Linear independence now follows, since the matrix of generalised quantum dimensions of a modular tensor category is related to the categorical $S$-matrix via elementary matrix operations, but the $S$-matrix is invertible.  For $m=0$ the left and right limits do not 
agree at $\epsilon=0$ but instead we can use a different point (still on the
same strip). 
Again we get a ring homomorphism whose image is isomorphic to the Virasoro fusion ring.

For (3) and (4),  the previous argument with limits apply and we get a homomorphism from the fusion ring. 
Recall the definition of $t$ via $m=tp_-$ (case (3)) and $m=tp_+$ (case (4)).
Finally, to prove that the image is isomorphic to the corresponding $A_1^{(1)}$ fusion ring we
note that $t$ takes only values in  $\{1, \dots, p_+-1, p_++1, \dots, 2p_+-1\}$  (case (3)) or $\{1,
\dots, p_--1, p_-+1,\dots, 2p_--1\}$  (case (4)). So that the quantum dimensions in this case span
at most a $p_+-1$ respectively $p_--1$ dimensional vector space. 
The same reasoning as in Proposition \ref{fusion-1p} applies here as well so the proof follows.
\end{proof}

\begin{remark} There is a purely geometrical interpretation of
  Theorem \ref{injective}.
  Because every $\mathbb{C}$-algebra homomorphism between finitely generated reduced $\mathbb{C}$-algebras is the
  pull-back of the corresponding regular map of the corresponding affine
  varieties, the observed surjective homomorphisms of rings 
  from $\re(\epsilon)>B_\epsilon^{p_+p_-}$ to (parts of)  $\re(\epsilon)<B_\epsilon^{p_+p_-}$ corresponds simply
  to an embedding of finite number of special points (zero dimensional
  variety) to $\mathcal{X}_{p,p'}$ (or $\mathcal{X}_{1,p}$).
\end{remark}

\begin{remark}
The varieties $\mathcal{X}_{p_+,p_-}$ (more precisely, the plane curve defined by $T_{p_+}(x)-T_{p_-}(y)=0$), have appeared in the physics literature in the problem of $(p_+, p_-)$ minimal string theory \cite{SS}. Points on the fusion variety $\mathcal{X}_{p_+,p_-}$ are in one-to-one correspondence with branes 
of the $(p_+, p_-)$ minimal string theory and the uniformisation parameter is expressible in terms of the boundary cosmological constant of the string theory. 

\end{remark}

\begin{remark} In \cite{BM}, the full asymptotic expansion of $\ch{X}(\tau)$
  was studied for atypical modules $X$ as $\tau \rightarrow 0^+$. In
  particular, by using a completely different method, the second author and K. Bringmann
  obtained results on unregularised quantum dimensions. More about the geometry
  of fusion varieties including higher rank vertex operator algebras \cite{BM2}
  will be a subject of \cite{CM2}.
\end{remark}

\section{Quantum modularity and quantum dimensions}\label{sec:qmod}

In this section, we associate to each singlet vertex operator algebra  a vector-valued
quantum modular form such that its associated $S$-matrix is the $S$-matrix
that arises in the $\re(\epsilon)<B_\epsilon^{p_+p_-}$ regime studied earlier (see Theorem
\ref{injective}). In other words, as explained in the introduction, we provide
another approach to ``semi-simplification'' from the limiting properties of
characters as we approach the real axis. 

We recall the definition of quantum modular forms, originally due to Zagier \cite{Za}.
\begin{definition}
A weight $k$ quantum modular 
form is a complex-valued function $f$ on $\mathbb{Q}$, or possibly smaller infinite set $\mathbb{Q} \setminus S$,
the so-called  {\em quantum set},  such that the function 
\begin{displaymath}
h_\gamma(x):=f(x)- \epsilon_\gamma(cx+d)^{-k}f(\frac{ax+b}{cx+d}), \ \ x \in \mathbb{Q} \setminus S
\end{displaymath}
$\gamma=\left(\begin{array}{cc} a & b \\ c & d \end{array} \right) \in SL(2,\mathbb{Z})$,
satisfies a suitable property of continuity and analyticity.
\end{definition}
Our understanding is that this definition is intentionally ambiguous so that it can accommodate more examples.
But all known examples of quantum modular forms enjoy slightly nicer properties. For {\em strong} quantum modular forms, there is an additional requirement that 
such a function extends to an analytic function defined in the both upper and
the lower half-plane and that limits coincide to all orders at roots of unity. This way one obtains a fairly non-standard object:  an analytic function in the upper half-plane which leaks throughout the { quantum  set} (typically a proper subset of $\mathbb{Q}$) into the lower half-plane. 

\subsection{$\SingVOA{p}$-models}

For brevity, let
\begin{displaymath}
F_{j, p}(\tau):= \sum_{n\in\Z} \sgn(n) q^{p\left(n+\frac{j}{2p}\right)^2},
\end{displaymath}
where we use the convention $\sgn(n)=1$, for $n \geq 0$ and $-1$ otherwise.
We are only interested in the (quotient) space of characters moded out by the
subspace of characters of standard modules. As may easily be seen, the
character $\ch{M_{r,s}}(\tau)$ admits a unique decomposition
$q_{r,s}(\tau)+\ch{M_{1,s'}}(\tau)$ where $q_{r,s}(\tau)$ is a finite $q$-series 
divided by the Dedekind eta function.  So we are left with $\ch{M_{1,s}}$,
where $1 \leq s \leq p-1$. One can easily verify that
\begin{displaymath}
\ch{M_{1,s}}(\tau)=\frac{F_{p-s,p}( \tau)}{\eta(\tau)}.
\nonumber
\end{displaymath}
In \cite{BM} it was proved that $\ch{M_{1,s}}(\tau)$ is a quantum modular form with quantum set $\mathbb{Q}$. 
We recall the construction here. As in Zagier's paper \cite{Za}, define for $\tau\in\bar{\mathbb{H}}:=\{\tau\in\C; \im(\tau)<0\}$
the non-holomorphic Eichler integral
\begin{equation*}
F_{j, p}^\ast(\tau):=\sqrt{2i}\int_{\overline{\tau}}^{i\infty}\frac{f_{j, p}(z)}{\left(z-\tau \right)^{\frac12}}dz,
\end{equation*}
where $f_{j, p}(z):=\sum_{n\in\Z}\left(n+\frac{j}{2p}\right)q^{p\left(n+\frac{j}{2p}\right)^2}$.
A key step in the proof of Theorem 4.1 of \cite{BM} is to show that $F_{j, p}
(\tau)$ agrees for $\tau=\frac{h}{k}\in \mathbb{Q}$ with $F_{j, p}^\ast
(\tau)$ up to infinite order (in the sense that appropriate limits agree) and
that $F_{j, p}^\ast (\tau)$ satisfies nice transformation law under a particular 
congruence subgroup.  Here we are interested in transformation properties 
under the full modular group $\Gamma(1)$.

We shall ignore the $\eta$ factor here as it does not affect the $S$-matrix. 
We form the vector valued function ${\bf F}(\tau)=\left(F_{1,p}(\tau),....F_{p-1,p}(\tau) \right)$. To find 
a quantum $S$-matrix we consider ${\bf F}^*(\tau)=(F^*_{1,p}(\tau),...,F^*_{p-1,p}(\tau))$.
Then we get the following transformation formula
\begin{align*}
f_{j, p}(-1/\tau) &= (-\tau) \sqrt{\frac{-i \tau}{2p}} \sum_{c=1}^{2p-1} e^{- \pi i j c/p}  f_{c,p}(\tau) \\
& = (-\tau) \sqrt{\frac{-i \tau}{2p}} (-2 i) \sum_{c=1}^{p-1} \sin\left(\frac{c j \pi}{p}\right) f_{c,p}(\tau).
\end{align*}
where we used the relation $f_{j,p}(\tau)=-f_{2p-j,p}(\tau)$. We now get
\begin{align*}
F_{j, p}^\ast(-1/\tau)&=\sqrt{2i}\int_{-1/\overline{\tau}}^{i\infty}\frac{f_{j, p}(z)}{\left(z+1/\tau \right)^{\frac12}}dz =\sqrt{2i} \int_{\overline{{\tau}}}^{0} \frac{u^{-2} f_{j,p}(-1/u)}{\left(-1/u+1/\tau \right)^{\frac12}}du \\
& =-\sqrt{2i} \int_{0}^{\bar{\tau}} \frac{\sqrt{u \tau}}{u^2 \sqrt{-\tau+u}}(-u) \sqrt{\frac{-i u}{2p}} \sum_{c=1}^{2p-1} e^{- \pi i j c/p}  f_{c,p}(u) du \\
& =-  \sqrt{\frac{2i \tau}{p}} \sum_{c=0}^{p-1} \sin\left(\frac{\pi c j}{p}\right) \sqrt{2i} \int_{0}^{\bar{\tau}} \frac{f_{c,p}(u)}{\sqrt{u-\tau}} du,
\end{align*}
after we take $u=-1/z$, $\frac{du}{u^2}=dz$. Form the  $(p-1) \times (p-1)$ matrix ${\bf S}(p)$, where $[{\bf S}(p)]_{c,j}=\sqrt{\frac{2}{p}} \sin(\frac{\pi c j}{p})$. 

\begin{thm} \label{qm-1p}
For $w \in \mathbb{H} \cup \overline{\mathbb{H}} \cup \mathbb{Q}$, we have that ${\bf F}(w)$ is a weight $\frac{1}{2}$ vector-valued quantum modular form. In
particular, we have that
\begin{equation}
\sqrt{\frac{1}{ w i}} {\bf F}\left(-\frac{1}{w}\right)-[{\bf S}(p)] {\bf F}(w)=[{\bf S}(p)] g(w),
\end{equation}
where $g(w)=-\sqrt{2i} \int_{0}^{ i \infty} \frac{{f}(u)du}{\sqrt{u-w}}$, $f(u)=(f_{1,p}(u),...,f_{p-1,p}(u))$ and ${\bf S}(p)$ is the $S$-matrix of $A_1^{(1)}$ at level $p-2$.
\end{thm}

\begin{proof} This follows from \cite{BM} and the discussion above. It was already proved in \cite{BM}
that the quantum set of ${\bf F}$ is $\mathbb{Q}$ by computing radial limits
at each point in $\mathbb{Q}$. 
It is also explained in \cite{BM}  and elsewhere that $g_\alpha(w)$, $\alpha
\in \Gamma(1)$, is a smooth function for $\alpha \in \mathbb{R}$. Although $g(w)$
is a priori only defined in $\overline{\mathbb{H}}$, we may take any path $L$ connecting points 0 to $i \infty$. Then we
can holomorphically continue $g(w)$ to all of $\mathbb{C}  \setminus L$. Thus, we obtain a continuation of $g$
which is smooth on $\mathbb{R}$ and analytic on $\mathbb{R} \setminus \{0\}$.
\end{proof}

\subsection{$\SingVOA{p_+,p_-}$-models}

In \cite{BM}, it was shown that all atypical characters are mixed quantum
modular forms with quantum set $\mathbb{Q}$; see also \cite{BCR} where similar
quantities appear. In the spirit of the $\SingVOA{p}$ algebra we consider the vector
spaces spanned by characters of simple modules over the
$\SingVOA{p_+,p_-}$ algebra. Consider its quotient space modulo the subspace
spanned by characters of standard modules \(\FF{\lambda}\) and of the
irreducible atypical modules $\mathcal{I}^\nu_{r,s;n}$ and \(\VirIrr{r,s}\). By the
relations used in Section 5.1 of \cite{BM} and in previous sections, easy
inspection shows that for a spanning set of the quotient space of we can choose 
\begin{align*}
\tilde{\chi}_{r,s}(\tau)&:=\eta(\tau)\ch{\mathcal{I}_{r,s;0}}(\tau)-\left(\sum_{k=0}^\infty
  q^{p_- p_+\left(k+\frac{2p_+p_- +p_+ s+p_- r}{2 p_+ p_-}\right)^2}-
    \sum_{k=0}^\infty q^{p_- p_+\left(k+\frac{2p_+p_- -p_+ s+p_- r}{2 p_+ p_-}\right)^2}\right) \\
&=\quad \sum_{k\geq0}(k+1)\left( q^{p_-p_+\left(k+\frac{2p_+p_- -p_+ s-p_- r}{2 p_+ p_-}\right)^2}+ 
     q^{p_- p_+\left(k+\frac{2p_+p_- + p_+s+p_- r}{2 p_+ p_-}\right)^2} \right) \\
 & \quad \phantom{\sum_{k\geq0}(k+1)}\quad 
 \left. - q^{p_- p_+\left(k+\frac{2p_+p_- -p_+ s+p_- r}{2 p_+ p_-}\right)^2}
     - q^{p_- p_+\left(k+\frac{2p_+p_- +p_+ s-p_- r}{2 p_+ p_-}\right)^2}
   \right),
\end{align*}
where $(r,s)\in\KacT{p_+,p_-}$. This space is precisely $\frac{(p_+-1)(p_- -1)}{2}$-dimensional.

Similarly, we introduce for $\tau \in \bar{\mathbb{H}}$,
\begin{displaymath}
{\bf G}^*(\tau)=\sqrt{2i} \int_{\bar{\tau}}^{i \infty} \frac{{\bf f}(z)}{(z-\tau)^{3/2}} d z,
\end{displaymath}
where ${\bf f}(z)=(...,\eta(z) \ch{\mathcal{L}_{r,s}(z)},...)$,
\((r,s)\in\KacT{p_+,p_-}\) and the 
entries are suitably ordered. 
We also let 
\begin{displaymath}
{\bf G}(\tau)=\left(\ldots, \tilde{\chi}_{r,s}(\tau), \ldots\right),
\quad \tau \in \mathbb{H},
\end{displaymath}
where the \((r,s)\) are ordered as in \({\bf G}^*\).
Now we extend ${\bf G}$ as a function defined on $\mathbb{H} \cup \overline{\mathbb{H}}$, where 
we let ${\bf G}(w):={\bf G}^*(w)$ for $w \in \overline{\mathbb{H}}$.
\begin{thm} \label{32}
For $w \in \mathbb{H} \cup \overline{\mathbb{H}} \cup \mathbb{Q}$, we have that ${\bf G}(w)$ is a weight $\frac{3}{2}$ quantum vector-valued modular form. In
particular, we have that
\begin{equation}
\left(\frac{1}{wi}\right)^{3/2} {\bf G}(-\frac{1}{w})+[{\bf SM}(p_+,p_-)] {\bf G}(w)=[{\bf SM}(p_+,p_-)] g(w),
\end{equation}
where $g(w)=- \sqrt{2i} \int_{0}^{i \infty} \frac{{\bf f}(u)}{(u-w)^{3/2}} d u $.
and $[{\bf SM}(p_+,p_-)]$ is the $S$-matrix of $(p_+,p_-)$ minimal models.
\end{thm}

\begin{proof} It was  already proven in \cite{BM} that each component of ${\bf G}(z)$ is a quantum modular form. 
In particular, to prove that radial limits of ${\bf G}(z)$ and ${\bf G}^*(z)$
agree as $z$ approaches a rational value from each side, it is  
sufficient to decompose ${\bf G}^*(\tau)$  into a sum of 4 Eichler integrals
associated to partial thetas and to use results obtained in    
\cite{BCR} to compute these limits which are expressed in terms of
$L$-function values. The remaining computations follow as in the previous section by
using modular transformation properties of Eichler's integrals and those of the characters
of the minimal models. Real analyticity is handled as in Theorem \ref{qm-1p}. 

\end{proof}

\begin{remark} 
Of course, we could have worked with the ``more natural'' basis $\ch{\mathcal{I}_{r,s;0}}(\tau)$ instead of $\tilde{\chi}_{r,s}(\tau)$. But it turns out that $\eta(\tau) \ch{\mathcal{I}_{r,s;0}}(\tau)$ is a quantum modular 
form of mixed weight $\frac{3}{2}$ and $\frac{1}{2}$, so the formulation of
Theorem \ref{32} would be significantly messier.

We should also remark  that Theorem \ref{32} was independently discovered in
\cite{torus} in the context of quantum knot invariants.
\end{remark}

\hspace*{1cm}

\noindent Department of Mathematical and Statistical Sciences, University of Alberta,
Edmonton, Alberta  T6G 2G1, Canada. 
\emph{email: creutzig@ualberta.ca}

\hspace*{1cm}

\noindent Department of Mathematics and Statistics, SUNY-Albany, 1400 Washington Avenue, Albany, NY 12222, USA.
\emph{email: amilas@albany.edu}

\hspace*{1cm}

\noindent Department of Theoretical Physics, Research School of Physics and
Engineering; and Mathematical Sciences Institute; 
The Australian National University, Acton, ACT 2601, Australia. 
\emph{email:simon.wood@anu.edu.au}

\end{document}